\documentclass[11pt]{article}
\usepackage{amsmath,amssymb,amsthm}
\usepackage{graphicx,enumerate,a4wide}
\usepackage{pdfsync,epstopdf,comment,url,ulem}
\usepackage[x11names,usenames,dvipsnames,svgnames,table,rgb]{xcolor}
\numberwithin{equation}{section}

\newtheorem{theorem}{Theorem}[section]

\newtheorem{corollary}[theorem]{Corollary}

\newtheorem{lemma}[theorem]{Lemma}
\newtheorem{remark}[theorem]{Remark}

\newtheorem{assumption}{Assumption}

\newcommand{\cf}{{cf.\ }}

\newcommand{\nn}{\mathbb{N}} 
\newcommand{\norm}[1]{\left\Vert {#1} \right\Vert} 
\newcommand{\dom}[1]{\mathrm{dom}\,{#1}} 

\newcommand{\NN}{\mathbb{N}}

\newcommand{\act}[1]{\left\langle {#1} \right\rangle} 
\newcommand{\seq}[2]{\{{#1}_{{#2}}\}_{{#2} \in \mathbb{N}}}
\newcommand{\Seq}[2]{\{{#1}^{{#2}}\}_{{#2} \in \mathbb{N}}}

\newcommand{\argmin}{\mathrm{argmin}}
\newcommand{\diag}{\mathrm{diag}}

\newcommand{\bh}{{\bf h}}
\newcommand{\bp}{{\bf p}}

\newcommand{\bu}{{\bf u}}
\newcommand{\bv}{{\bf v}}
\newcommand{\bx}{{\bf x}}
\newcommand{\bbx}{{\bf X}}
\newcommand{\bbq}{{\bf Q}}
\newcommand{\balpha}{{\mbox{\boldmath $\alpha$}}}

\newcommand{\bo}{{\bf 0}}

\newcommand{\bbu}{{\bf U}}
\newcommand{\bbi}{{\bf I}}

\newcommand{\by}{{\bf y}}

\newcommand{\bb}{{\bf b}}

\newcommand{\bba}{{\bf A}}
\newcommand{\bbd}{{\bf D}}

\newcommand{\real}{\mathbb{R}} 

\newcommand{\ed}[1]{#1}
\definecolor{light-gray}{gray}{0.7}
\newcommand{\delete}[1]{#1}

\title{The Cyclic Block Conditional Gradient Method for \\Convex Optimization Problems}
\author{Amir Beck\thanks{Faculty of Industrial Engineering and Management, Technion, Haifa, Israel ({\tt \{becka,epauwels,ssabach\}@ie.technion.ac.il}).} \and Edouard Pauwels\footnotemark[1] \and Shoham Sabach\footnotemark[1]}
\date{\today}

\begin{document}
\maketitle
\begin{abstract}
	In this paper we study the convex problem of optimizing the sum of a smooth function and a compactly supported non-smooth term with a specific separable form. We analyze the block version of the generalized conditional gradient
	method when the blocks are chosen in a cyclic order. A global sublinear rate of
	convergence is established for two different stepsize strategies commonly used in this
	class of methods. Numerical comparisons of the proposed method to both the classical conditional
	gradient algorithm and its random block version demonstrate the
	effectiveness of the cyclic block update rule.
\end{abstract}


\section{Introduction} \label{Sec:Introduction}
With the growth of  size of problems commonly encountered in many applied fields, there
	is a strong demand for numerical methods featuring low computational cost iterative
	schemes. By low computational cost, we mean algorithms which require at most matrix
	by vector multiplication (inversion of matrices are, for example, too expensive). In this
	context, it is necessary to propose and analyze numerical schemes that
	\begin{itemize}
		\item are based on computationally efficient steps;
		\item exploit problem structure and data information;
		\item enjoy global convergence properties and iteration complexity estimates.
	\end{itemize}
	We consider structured convex problems consisting of minimizing the sum of two terms: a smooth term, which is a composition of a smooth function and a linear mapping and a nonsmooth separable term.   We focus on programs for which the geometry of the non-smooth part exhibits such a degree of complexity that proximal-based methods \cite{BT09,BTet2013,N13} do not constitute a viable alternative. Indeed, the efficiency of these methods considerably deteriorates in situations that do not fit in a ``favorable geometric settings'', see \cite{CJN2013,HJN} for a more detailed description of this concept. Algorithms based on linear oracles, such as the conditional gradient method (also known as the Frank-Wolfe algorithm)  \cite{FW1956,LP1966,DR1967,DH1978,J2013}  and their extensions to structured composite problems  \cite{BL2008,BLM2009,B2012}, or block separable problems \cite{LJJSP2013}, are based on the principle of iteratively solving linearized subproblems. In settings where computing the proximal operator is too expensive, this approach has proven to be competitive. Successful examples of applications which benefit from this approach include trace-norm constrained or penalized problems \cite{JS2010,CJN2013}  and structured multiclass classification with extremely large number of classes  \cite{LJJSP2013}. \ed{On the theoretical side, for most of the algorithms mentioned in the references above, a sublinear rate of $O(1/k)$, both in function values and duality gap, is available}.\\
\indent Continuously increasing problem dimensions have motivated the principle of taking advantage of available block structure of the problem at hand. This led to the development of variable decomposition methods which break down the 	original large-scale problem into several much smaller subproblems that could be solved efficiently.

 In recent works, \cite{N2012} and \cite{RT2014} analyzed the average case iteration complexity of block versions of gradient,  projected gradient and forward-backward methods where at each iteration, the block to be updated is selected at random.
We refer to such a block selection rule as the \textit{random update rule}. In
	the context of linear oracles, \cite{LJJSP2013} applied this approach to the conditional
	gradient method focusing on implementing it for the \ed{structured} Support Vector Machine (SVM)
	training problem (see more details in Section \ref{SSec:Numerics-SVM}). Analysis of
	algorithms involving random update rules typically provides average case complexity results.  A different kind of works consider updating the blocks
	in a cyclic order. We refer to such a deterministic update rule as the \textit{cyclic update rule}. In
	this context, \cite{LT1992} provides an asymptotic analysis of exact coordinate
	minimization for composite strongly convex problems. In \cite{BTet2013}, a global rate of convergence result was established for the cyclic block coordinate gradient projection method for convex problems over feasible sets with a separable structure.
 In \cite{saha2013nonasymptotic}, an explicit rate is given for the block proximal gradient method for $\ell_1$ regularized convex problems.
	 For this line of works, the estimates given for the cyclic update
	rule deterministically hold for the sequence of function values. As we already
	mentioned, this is not the case for the random update rule for which only average case estimates are available.  \\
	\indent Another relevant feature of block decomposition methods is the fact that they usually allow to take a substantillay larger step at each iteration (with respect to each block) compared to their classical non-block counterparts \ed{(using, for example, fixed stepsize, backtracking or line search)}. This fact potentially gives a numerical advantage to block decomposition algorithms compared to classical variants, see \cite{BTet2013}.\\
\indent 	In this work, we propose a cyclic block version of the generalized conditional gradient
	method \cite{BL2008,BLM2009,B2012} which for ease of reference is called the Cyclic Block
	Conditional Gradient (CBCG). \ed{The word ``generalized" means that we consider a more general nonsmooth part, which is not necessarily an indicator function. The word ``cyclic" means here that each block is updated once at each iteration. The order in which the blocks are updated may vary arbitrarily between iterations and our analysis therefore includes the random permutation approach.} We provide \ed{deterministic and global convergence rate estimates} for the predefined stepsize strategy \cite{DH1978,J2013}, and for an adaptive stepsize
	strategy \cite{LP1966} \ed{including a backtracking version of it. These rates hold independently of the order of updating the blocks at each iteration}. We also establish rate estimates for an optimality measure in
	the spirit of \cite{J2013,B2012}\ed{, which constitutes an additional novelty compared to the results presented in \cite{BTet2013}.}
\ed{All the rates proved below have the form of $O(1/k)$ where $k$ is the number of complete iterations (over all blocks). Interestingly, this analysis leads to new convergence results for methods that do not relate directly to linear oracles at first sight. For example, our results lead to explicit deterministic rate estimates in terms of the duality gap for the cyclic and random permutation variants of the Stochastic Dual Coordinate Ascent (SDCA) of \cite{shalev2013stochastic}.}
	
	We numerically compare the proposed CBCG method (using both the cyclic and random permutation updating rules) to its random update rule counterpart, that is, the Random Block Conditional Gradient (RBCG), which was proposed and
	analyzed in \cite{LJJSP2013}. Extensive simulations on a large number of synthetic examples
	suggest that CBCG is competitive with both RBCG and the classical conditional gradient
	algorithm (CG). Finally, we also compare CBCG and RBCG on the problem of training the
	\ed{structured} SVM \cite{TGK2004,TJHA2005} for the optical character recognition task (OCR)
	originally proposed in \cite{TGK2004}. \ed{In this setting, we observe that the random permutation updating rule has advantage over the other updating rules.}\\
\indent  The next section is dedicated to the presentation of the model and main assumptions
	(see Section \ref{SSec:Model}) and to the description of the CBCG algorithm (see Section
	\ref{SSec:Algorithm}). Section \ref{Sec:OptMeas} presents few auxiliary results for an
	optimality measure which is commonly encountered when using methods which are based on linear
	oracles. In Section \ref{Sec:Analysis} we present our theoretical findings about the rate
	of convergence results of the CBCG method. We split this section into two subsections
	which deal with the two different stepsize strategies that we analyze in this paper
	(Section \ref{SSec:Pred} for the predefined stepsize and Section \ref{SSec:Adap} for the
	adaptive stepsize). We conclude Section \ref{Sec:Analysis} with a discussion on a
	backtracking version of the CBCG method (see Section \ref{SSec:Back}). \ed{Section \ref{Sec:ML} presents an extension of the analysis given in Section \ref{SSec:Adap} for the specific case where the stepsize is chosen using exact line search for problems in which the smooth part of the objective function is quadratic. This leads to a discussion about the implications for block coordinate descent methods.} Numerical
	experiments on synthetic data and on the \ed{structured} SVM training problem are presented in
	Section \ref{Sec:Numerics}.
\medskip

	{\bf Conventions.} Throughout the paper the underlying vector space is the $n$-dimensional
	Euclidean space $\real^{n}$ with the $l_2$-norm which is denoted by $\norm{\cdot}$. This notation is also used for the matrix norm, which is assumed to be the spectral norm. We
	will consider the partition of an arbitrary vector $\bx \in \real^{n}$ into $N$ blocks
	where each block consists of a subset of the $n$ coordinates. The size of each block (the
	number of coordinates) is given by the integer $n_{i}$ for $i = 1 , 2 , \ldots , N$, such
	that $\sum_{i = 1}^{N} n_{i} = n$. The $i$th block of a vector $\bx \in \real^{n}$ is
	denoted by $\bx_i$. We assume that $\bx \in \real^{n}$ can be written as follows
	\begin{equation*}
		\bx =
			\begin{pmatrix}
			  	\bx_1  \\
			  	\bx_2  \\
			  	\vdots \\
			  	\bx_N
			  \end{pmatrix}.
	\end{equation*}
	For any $i = 1 , 2 , \ldots , N$ we define the matrix $\bbu_{i} \in \real^{n \times n_{i}}$ as
	the sub-matrix of the $n \times n$ identity matrix consisting of the columns corresponding
	to the $i$th block. Thus, in particular,
	\begin{equation*}
		\left(\bbu_{1} , \bbu_{2} , \ldots , \bbu_{N}\right) = \bbi_{n}.
	\end{equation*}
	It is clear that using these notations, we have for any $\bx \in \real^{n}$ that $\bx_{i} =
	\bbu_{i}^{T}\bx$ and $\bx = \sum_{i = 1}^{N} \bbu_{i}\bx_{i}$. Finally, for any subset $S$ of $
	\real^{n}$, $\delta_{S}$ denotes the indicator function of $S$ which takes the value $0$ on
	$S$ and $+\infty$ otherwise.
	
\section{The Optimization Model and Algorithm} \label{Sec:OptModel}
\subsection{Problem Formulation and Assumptions} \label{SSec:Model}
	We consider the optimization model
	\begin{equation} \label{mainproblem}
		\min_{\bx \in \real^{n}} \left\{ H\left(\bx\right) \equiv F(\bba \bx) +
		  \sum_{i = 1}^{N} g_{i}(\bx_{i}) \right\},
	\end{equation}
	where $\bba \in \real^{m \times n}$. We make the following standing
	assumption on model \eqref{mainproblem}.
	\begin{assumption} \label{AssumptionA}
		$ $ \\ \vspace{-0.2in}
		\begin{itemize}
			\item[$\rm{(i)}$] $g_{i} : \real^{n_{i}} \rightarrow \left(-\infty , \infty\right]
				$, $i = 1 , 2 , \ldots , N$, is a proper, closed and convex
				function which satisfies
				\begin{itemize}
					\item $X_{i} \equiv \dom{g_{i}} \subseteq \real^{n_{i}}$ is a compact set
						with diameter $D_{i}$, that is,
						\begin{equation*}
							\sup_{\bx_{i} , \by_{i} \in X_{i}} \norm{\bx_{i} - \by_{i}} =
							D_{i}.
						\end{equation*}
					\item $g_{i}$ is globally Lipschitz on $X_{i}$ with constant $l_{i}$, that
						is,
						\begin{equation}
							\left| g_{i}\left(\bx_{i}\right) - g_{i}\left(\by_{i}\right)
							\right| \leq l_{i}\norm{\bx_{i} - \by_{i}}, \quad \forall \,\,
							\bx_{i} , \by_{i} \in X_{i}.
						\end{equation}
				\end{itemize}
			\item[$\rm{(ii)}$] $F : \real^{m} \rightarrow \ed{\real}$ is
				convex and continuously differentiable\footnote{\ed{A function is continuously differentiable over a given set $D$ if it is continuously differentiable over an open set containing $D$.}}  over $\bba\left(X_{1} \times X_{2} \times
				\cdots \times X_{N}\right) \subseteq \real^{m}$ and has Lipschitz continuous
				gradient with constant $L_{F}$, that is,
				\begin{equation*}
					\norm{\nabla F\left(\bx\right) - \nabla F\left(\by\right)} \leq L_{F}
					\norm{\bx - \by}, \quad \forall \,\, \bx , \by \in \bba\left(X_{1} \times
					X_{2} \times \cdots \times X_{N}\right).
				\end{equation*}
		\end{itemize}
	\end{assumption}
	Since $g_{i}$ is assumed to be convex, it immediately implies
	that the domain $X_{i}$ is a convex subset of $\real^{n_{i}}$ for $i = 1 , 2 , \ldots , N$. We set $g
	(\bx) \equiv \sum_{i = 1}^{N} g_{i}(\bx_i)$ and $f(\bx)\equiv F(\bba \bx)$. The domain of $g$ is denoted by $\dom{g}
	\equiv X$ and its diameter is denoted by $D$. Using these simplified notations, problem
	\eqref{mainproblem} actually consists of minimizing the sum $f + g$. It holds that $X \equiv
	X_{1} \times X_{2} \times \cdots \times X_{N}$ and therefore
	\begin{equation} \label{sumd}
		D^{2} = \sum_{i = 1}^{N} D_{i}^{2}.
	\end{equation}
	\begin{remark} \label{rem:tradiCG}
		By setting $g\left(\cdot\right) = \delta_{X}\left(\cdot\right)$, we recover the
		constrained optimization model that motivated the development of the traditional
		conditional gradient method (see \cite{J2013} and references therein). This is also the case when we add a linear
		term to the indicator.
	\end{remark}
	Under Assumption \ref{AssumptionA}, problem \eqref{mainproblem} is
	guaranteed to attain its optimal value, and therefore the optimal set, which is denoted by
	$X^{\ast}$, is nonempty and the corresponding optimal value is denoted by $H^{\ast} \in
	\real$. For each block $i \in \left\{ 1 , 2 , \ldots , N \right\}$, we employ the following
	notation:
	\begin{itemize}
		\item $\bba_{i} \equiv \bba \bbu_i$, so that $\bba = \left(\bba_{1} , \bba_{2} , \ldots , \bba_{N}\right)$;
		\item $\nabla_i f(\bx) \equiv \bbu_i^T \nabla f(\bx)$ denotes the partial gradient of $f$.
	\end{itemize}
	Using the previous notations, we have that $\nabla_{i} f\left(\bx\right) = \bba_{i}^{T}\nabla
	F(\bba \bx)$. We will use a refined notion of Lipschitz continuity that fits our block separable composite setting. This is expressed by
	 the following standing assumption.
	\begin{assumption} \label{AssumptionB}
		For each $i = 1 , 2 , \ldots , N$, there exists a constants $\beta_{i} > 0$, such that for any $
		\bx \in X$ and any $\bh_{i} \in \real^{n_{i}}$ satisfying $\bx + \bbu_{i}\bh_{i} \in
		X$, it holds that
		\begin{equation*}
			\norm{\nabla F(\bba\bx + \bba_{i}\bh_{i}) - \nabla F(\bba\bx)} \leq
			\beta_{i}\norm{\bba_{i}\bh_{i}}.
		\end{equation*}
	\end{assumption}
	Assumption \ref{AssumptionB} can be seen as a consequence of Assumption \ref{AssumptionA}.
	Indeed, it is always possible to set $\beta_{i} = L_{F}$, $i = 1 , 2 , \ldots , N$.
	However, adopting this more refined convention provides additional algorithmic flexibility which allows to take advantage of conditioning
	disparities between different blocks \ed{by using stepsizes which are functions of the Lipschitz constants of the blocks, rather than the global Lipschitz constant}. See for example Section  \ref{SSec:Exact}, where it is shown how this approach allows the usage of exact line search for quadratic problems. We also note that when $\bba = \bbi$, then $\beta_i$ can be chosen to be the $i$th block Lipschitz constant of the gradient of $F$ (see, e.g., \cite{BTet2013,B15}), which is always a quantity smaller than $L_F$. We define the following quantity
	\begin{equation} \label{def_beta_min}
		\beta_{\min} \equiv \min\left\{ \beta_{1} , \beta_{2} , \ldots , \beta_{N} \right\} >
		0.
	\end{equation}
	The following important result will play a central role in the forthcoming analysis. The proof is almost identical to the well known proof of the descent lemma (see \cite{B99}), and is thus given in Appendix \ref{sec:app}.
	\begin{lemma}[Composite block descent lemma] \label{L:BlockDescent}
		Let $i \in \left\{ 1 , 2 , \ldots , N
		\right\}$, then for any $\bx \in X$ and $\bh_{i} \in \real^{n_{i}}$ such that $\bx
		+ \bbu_{i}\bh_{i} \in X$, we have
		\begin{equation} \label{L:BlockDescent:1}
			f(\bx + \bbu_{i}\bh_{i}) \leq f(\bx) + \act{\nabla_{i} f(\bx
			), \bh_{i}} + \frac{\beta_{i}}{2}\norm{\bba_{i}\bh_{i}}^2.
		\end{equation}
	\end{lemma}
\subsection{The Cyclic Block Conditional Gradient (CBCG) Method} \label{SSec:Algorithm}
	The generalized conditional gradient method \cite{BL2008,BLM2009,B2012} can be applied to
	problem \eqref{mainproblem} when the corresponding linear oracle is available. Therefore we assume that for
	any $\bx \in X$, the solution of the following problem can be easily computed:
	\begin{equation*}
		\min_{\bv \in X} \left\{ \act{\nabla f\left(\bx\right) , \bv} + g\left(\bv\right)
		\right\}.
	\end{equation*}
	We exploit here the separability of the function $g$ (see Section \ref{SSec:Model}) and
	propose a block decomposition extension of the generalized conditional gradient method
	which we call the Cyclic Block Conditional Gradient (CBCG) method. Before stating the
	algorithm, we will need the following additional notation. Let $\Seq{\bx}{k}$ be a given
	sequence, then for any $i = 1 , 2, \ldots , N$, we define
	\begin{equation} \label{BlockUpdate}
		\bx^{k,i} =
			\begin{pmatrix}
				\bx_{1}^{k + 1} \\
				\vdots \\
				\bx_{i}^{k + 1} \\
				\bx_{i + 1}^{k} \\
				\vdots \\
				\bx_{N}^{k}
			\end{pmatrix}.
	\end{equation}
	That is, the first $i$ blocks in $\bx^{k,i}$ are those of $\bx^{k+1}$ and the remaining $N-i$ blocks are those of $\bx^k$.
     It is clear that using this notation we have  $\bx^{k,0} = \bx^{k}$ and $
	\bx^{k + 1} = \bx^{k,N}$. The algorithm is given now.
\bigskip

    \fcolorbox{black}{Ivory2}{\parbox{15cm}{{\bf CBCG: Cyclic Block Conditional Gradient} \\
    {\bf Initialization.} $\bx^{0} \in X$ and $\alpha_{i}^{k} \in \left[0 , 1\right]$ for all
    		$k \in \nn$ and $i = 1  ,2, \ldots , N$. \\
    {\bf General Step.} For $k = 1 , 2 , \ldots$,
        \begin{itemize}
					\item[$\rm{(i)}$] For any $i = 1 , 2 , \ldots , N$, compute
						\begin{equation} \label{CBCG:1}
							\bp_{i}^{k} \in \argmin_{\bp_{i} \in X_{i}} \left\{ \langle \nabla_{i}
							f(\bx^{k,i - 1}) , \bp_{i} \rangle  + g_{i}(\bp_{i}) \right
							\},
						\end{equation}
						and then
						\begin{equation} \label{CBCG:2}
							\bx^{k,i} = \bx^{k,i - 1} + \alpha_{i}^{k}\bbu_{i}(\bp_{i}^{k} -
							\bx^{k,i - 1}_{i}).
						\end{equation}
					\item[$\rm{(ii)}$] Update $\bx^{k + 1} = \bx^{k,N}$.
				\end{itemize}
		}
	}
\bigskip

	We will first analyze, in Section \ref{SSec:Pred}, the convergence rate of the CBCG method
	using a predefined stepsize \cite{DH1978}. Here, the predefined stepsize that we use is given, for any $i = 1 , 2 , \ldots , N$, by
	\begin{equation*}
		\alpha_{i}^{k} = \alpha^{k} \equiv \frac{2}{k + 2}.
	\end{equation*}
	In Section \ref{SSec:Adap}, we will consider an adaptive stepsize rule \cite{LP1966}, which is
	determined by the minimization of the quadratic upper bound of $H$ related to
	\eqref{L:BlockDescent:1}. The expression of this stepsize will be made precise below (see
	\eqref{adaptive}). The backtracking variant of the CBCG with adaptive stepsize rule is presented and analyzed in Section \ref{SSec:Back}.
	
\section{The Optimality Measure} \label{Sec:OptMeas}
	In this section, we describe few properties of an optimality measure including its block counterparts. This measure is typical when discussing
	methods which are based on linear oracles and usually plays a crucial role in the
	convergence analysis, see \cite{BT2004} for an overview and \cite{LJJSP2013,B2012} for a link with Fenchel duality.  For any $\bx
	\in X$, we define the following quantity
	\begin{equation} \label{D:Px}
		p(\bx) \in \argmin_{\bp \in X} \left\{ \langle \nabla f(\bx) , \bp\rangle
		+ g(\bp) \right \},
	\end{equation}
	as well as the optimality measure
	\begin{equation} \label{D:OptMea}
		S(\bx) \equiv \max_{\bp \in X} \left\{ \langle \nabla f(\bx) , \bx -
		\bp\rangle  + g(\bx) - g(\bp) \right\} = \langle \nabla f(\bx) ,
		\bx - p(\bx)\rangle  + g(\bx) - g(p(\bx)),
	\end{equation}
	where the last equality follows from \eqref{D:Px}. The function $S$ is an optimality
	measure in the sense that it is non-negative on $X$, and it is zero only on $X^{\ast}$. Furthermore, for any $\bx \in X$, the quantity $S\left(\bx\right)$ is an upper bound on $H(\bx)-H^*$ as stated in the following lemma whose short proof is given for the sake of completeness.
	\begin{lemma} \label{L:UppBoundS} $S(\bx) \geq H\left(\bx\right) - H^{\ast}$.
	\end{lemma}
	\begin{proof}
		For any $\bx^{\ast} \in X^{\ast}$ we have
		\begin{align*}
			S(\bx) & = \act{\nabla f(\bx) , \bx - p(\bx)} +
			g(\bx) - g(p(\bx)) \\
			& = \act{\nabla f(\bx) , \bx} + g(\bx)- \left [\act{\nabla f
			(\bx) , p(\bx)} + g(p(\bx))\right ] \\
			& \geq \act{\nabla f(\bx) , \bx} + g(\bx)- \left [\act{\nabla
			f(\bx) , \bx^{\ast}} + g(\bx^{\ast})\right ] \\
			& = \act{\nabla f(\bx) , \bx - \bx^{\ast}} + g(\bx) -
			g(\bx^{\ast}),
		\end{align*}
		where the inequality follows from \eqref{D:Px}. Using the convexity of $f$, we obtain that
		\begin{equation*}
			S(\bx) \geq \act{\nabla f(\bx) , \bx - \bx^{\ast}} +
			g(\bx) - g(\bx^{\ast}) \geq f(\bx) -
			f(\bx^{\ast}) + g(\bx) - g(\bx^{\ast}) =
			H(\bx) - H^{\ast}.
		\end{equation*}
		This proves the desired result.
	\end{proof}
We refine the notations introduced in \eqref{D:Px} and \eqref{D:OptMea} in order to fit them to our block structured setting.
 For any $\bx \in X$ and any $i = 1 , 2 , \ldots , N$,
	we set
	\begin{equation} \label{D:Pxi}
		p_{i}(\bx) \in \argmin_{\bp_{i} \in X_{i}} \left\{ \act{\nabla_{i} f
		(\bx) , \bp_{i}} + g_{i}(\bp_{i}) \right \},
	\end{equation}
	and  define the block optimality measure, \ed{ which was already introduced in \cite{LJJSP2013} when $g_{i} \equiv 0$,}
	\begin{equation} \label{D:OptMeai}
		S_{i}(\bx) \equiv \max_{\bp_{i} \in X_{i}} \left\{ \act{\nabla f_{i}
		(\bx) , \bx_{i} - \bp_{i}} + g_{i}(\bx_{i}) - g_{i}(\bp_{i}
		) \right\}.
	\end{equation}		
	It is clear that in this case it is also true that
	\begin{equation} \label{FactOptMe}
		S_{i}\left(\bx\right)  = \act{\nabla f_{i}\left(\bx\right) , \bx_{i} - p_{i}\left(\bx
		\right)} + g_{i}\left(\bx_{i}\right) - g_{i}\left(p_{i}\left(\bx\right)\right).
	\end{equation}
	Using the separability of both $g$ and $X$, we have for any $\bx \in X$ that
	\begin{equation} \label{sum_S}
		S\left(\bx\right) = \sum_{i = 1}^{N} S_{i}\left(\bx\right).
	\end{equation}
	There might be multiple optimal solutions for problem \eqref{D:Px} and also for
	problem \eqref{D:Pxi}. Our only assumption is that the choices
of $p_1(\bx),p_2(\bx),\ldots,p_N(\bx)$ and $p(\bx)$ are made under the restriction that
	\begin{equation*}
		p(\bx) =
			\begin{pmatrix}
				p_{1}(\bx) \\
				p_{2}(\bx) \\
				\vdots \\
				p_{N}(\bx)
			\end{pmatrix}.
	\end{equation*}
	The following Lipschitz-type property of the block optimality measure $S_{i}$, $i = 1 , 2 ,
	\ldots , N$, will be crucial in the forthcoming analysis.
	\begin{lemma} \label{L:LipSi}
		Let $\bx , \by \in X$ be two vectors
		which satisfy $\bx_{i} = \by_{i}$ for some $i = 1 , 2 , \ldots , N$. Then the
		following inequality holds
		\begin{equation*}
			\left|S_{i}(\bx) - S_{i}(\by)\right| \leq L_{F}D_{i}
			\norm{\bba_{i}} \cdot \norm{\bba(\bx - \by)}    .
		\end{equation*}
	\end{lemma}
	\begin{proof}
		From \eqref{FactOptMe} we have
		\begin{align}		
			S_{i}(\bx) & = \act{\nabla f_{i}\left(\bx\right) , \bx_{i} - p_{i}(\bx)} + g_{i}(\bx_{i}) - g_{i}(p_{i}(\bx)) \label{L:LipSi:1} \\
			& = \act{\nabla f_{i}(\by) , \bx_{i} - p_{i}(\bx)} + g_{i}
			(\bx_{i}) - g_{i}(p_{i}(\bx)) + \act{\nabla f_{i}
			(\bx) -  \nabla f_{i}(\by), \bx_{i} - p_{i}(\bx)}.
			\nonumber
		\end{align}
		Now, using the fact that $f\left(\bx\right) \equiv F\left(A\bx\right)$  and Assumption \ref{AssumptionA}(ii), we obtain
		\begin{align}		
			\act{\nabla f_{i}(\bx) -  \nabla f_{i}(\by), \bx_{i} - p_{i}
			(\bx)} & = \act{\bba_{i}^{T}(\nabla F(\bba\bx) -  \nabla F
			(\bba\by)) , \bx_{i} - p_{i}(\bx)} \nonumber \\
			& = \act{\nabla F(\bba\bx) -  \nabla F(\bba\by) , \bba_{i}
			(\bx_{i} - p_{i}(\bx))} \nonumber \\
			& \leq \norm{\nabla F(\bba\bx) -  \nabla F(\bba\by)} \cdot
			\norm{\bba_{i}(\bx_{i} - p_{i}(\bx))} \nonumber\\
			& \leq L_{F}\norm{\bba(\bx - \by)} \cdot \norm{\bba_{i}(\bx_{i} - p_{i}
			(\bx))} \nonumber \\
			& \leq L_{F}D_{i}\norm{\bba_{i}} \cdot \norm{\bba(\bx - \by)},
			\label{L:LipSi:2}
		\end{align}
		where the first inequality follows from the Cauchy-Schwarz inequality and the last
		inequality follows from the fact that both $\bx_{i}$ and $p_{i}(\bx)$
		belong to $X_{i}$ (see Assumption \ref{AssumptionA}(i)). Finally, by combining
		\eqref{L:LipSi:1} with \eqref{L:LipSi:2}, and using the fact that $\bx_{i} = \by_{i}$,
		we obtain that
		\begin{align}
			S_{i}(\bx) & \leq \act{\nabla f_{i}(\by) , \bx_{i} - p_{i}
			(\bx)} + g_{i}(\bx_{i}) - g_{i}(p_{i}(\bx)
			) + L_{F}D_{i}\norm{\bba_{i}} \cdot \norm{\bba(\bx - \by)} \nonumber \\
			& = \act{\nabla f_{i}(\by) , \by_{i} - p_{i}
			(\bx)} + g_{i}(\by_{i}) - g_{i}(p_{i}(\bx)
			) + L_{F}D_{i}\norm{\bba_{i}} \cdot \norm{\bba(\bx - \by)} \nonumber \\
			& \leq S_{i}(\by) + L_{F}D_{i}\norm{\bba_{i}} \cdot \norm{\bba(\bx - \by
			)}, \label{L:LipSi:3}
		\end{align}
		where the last inequality follows from the definition of $S_{i}$ (see \eqref{D:OptMeai}).
		Changing the roles of $\bx$ and $\by$, we also obtain that
		\begin{equation*}
			S_{i}(\by) \leq S_{i}(\bx) +L_{F}D_{i}\norm{\bba_{i}} \cdot
			\norm{\bba(\bx - \by)},
		\end{equation*}
		which along with \eqref{L:LipSi:3} yields the desired result.
	\end{proof}

\section{Convergence Analysis of the CBCG Method} \label{Sec:Analysis}
	This section is devoted to the convergence analysis of the CBCG algorithm. We will first
	prove, in Section \ref{SSec:Pred}, a sublinear convergence rate for the variant with the predefined
	stepsize rule. A similar rate of convergence will be
	then established, in Section \ref{SSec:Adap}, for the variant with the adaptive stepsize rule. Finally, we describe a backtracking procedure in Section
	\ref{SSec:Back} which allows to use the CBCG method when the constants $\beta_{i}$, $i = 1 , 2 ,
	\ldots , N$, given in Assumption
	\ref{AssumptionB} are unknown in advance. We begin with an extension of Lemma
	\ref{L:BlockDescent} that holds for any choice of stepsize.
	\begin{lemma} \label{L:Descent}
		Let $\Seq{\bx}{k}$ be the sequence generated by the CBCG method. Then, for any $k \geq 0$ and
		$i \in \left\{ 1 , 2 , \ldots , N \right\}$, we have
		\begin{equation*}
			H(\bx^{k,i}) \leq H(\bx^{k,i - 1}) - \alpha_{i}^{k}
			S_{i}(\bx^{k,i - 1}) + \frac{(\alpha_{i}^{k})^{2}\beta_{i}}
			{2}\|\bba_{i}(\bp_{i}^{k} - \bx_{i}^{k})\|^{2}.
		\end{equation*}
	\end{lemma}
	\begin{proof}
		First, by the definition of the main step of the CBCG method (see \eqref{CBCG:2}), we
		have
		\begin{align*}
			H(\bx^{k,i}) & = f(\bx^{k,i}) + g(\bx^{k,i}) \\
			& = f(\bx^{k,i - 1} + \alpha_{i}^{k}\bbu_{i}(\bp_{i}^{k} - \bx_{i}^{k,i - 1}
			)) + g(\bx^{k,i - 1} + \alpha_{i}^{k}\bbu_{i}(\bp_{i}^{k} -
			\bx_{i}^{k,i - 1})).
		\end{align*}
		We can now use Lemma \ref{L:BlockDescent} to obtain
		\begin{align}
			H(\bx^{k,i}) & \leq f(\bx^{k,i - 1}) + \alpha_{i}^{k}
			\langle \nabla_{i} f(\bx^{k,i - 1}) , \bp_{i}^{k} - \bx_{i}^{k,i - 1}\rangle
			\nonumber \\
			& + \frac{(\alpha_{i}^{k})^{2}\beta_{i}}{2}\|\bba_{i}(\bp_{i}^{k}
			- \bx_{i}^{k, i - 1})\|^{2} + g (\bx^{k,i - 1} + \alpha_{i}^{k}\bbu_{i}
			(\bp_{i}^{k} - \bx_{i}^{k,i - 1}) ). \label{L:Descent:1}
		\end{align}
		The last term can be bounded from above as follows:
		\begin{align}
			g(\bx^{k,i - 1} + \alpha_{i}^{k}\bbu_{i}(\bp_{i}^{k} - \bx_{i}^{k,i - 1}
			)) & = \sum_{j = 1 , j \neq i}^{N} g_{j}(\bx_{j}^{k,i - 1})
			+ g_{i}((1 - \alpha_{i}^{k})\bx_{i}^{k,i - 1} + \alpha_{i}^{k}
			\bp_{i}^{k}) \nonumber \\
			& \leq \sum_{j = 1 , j \neq i}^{N} g_{j}(\bx_{j}^{k,i - 1})
			+ (1 - \alpha_{i}^{k})g_{i}(\bx_{i}^{k,i - 1}) + \alpha_{i}
			^{k}g_{i}(\bp_{i}^{k}) \nonumber \\
			& = g(\bx^{k,i - 1}) + \alpha_{i}^{k}(g_{i}(\bp_{i}^{k})
			- g_{i}(\bx_{i}^{k,i - 1})), \label{L:Descent:2}
		\end{align}
		where the inequality follows from the convexity of each $g_{i}$, $i = 1 , 2 , \ldots ,
		N$. Now, combining \eqref{L:Descent:1} with \eqref{L:Descent:2} and using the
		definition of $S_{i}$ (see \eqref{FactOptMe}), we obtain
		\begin{align*}
	 		H(\bx^{k,i}) & \leq f(\bx^{k,i - 1}) + g(\bx^{k,i - 1}
	 		) - \alpha_{i}^{k}(\langle \nabla_{i} f(\bx^{k,i - 1}) , \bx_{i}
	 		^{k,i - 1} - \bp_{i}^{k}\rangle  + g_{i}(\bx_{i}^{k,i - 1}) - g_{i}
	 		(\bp_{i}^{k})) \\
	 		& + \frac{(\alpha_{i}^{k})^{2}\beta_{i}}{2}\|\bba_{i}(\bp_{i}^{k}
	 		- \bx_{i}^{k, i - 1})\|^{2} \\
			& = H(\bx^{k,i - 1}) - \alpha_{i}^{k}S_{i}(\bx^{k,i - 1}) +
			\frac{(\alpha_{i}^{k})^{2}\beta_{i}}{2}\|\bba_{i}(\bp_{i}^{k} -
			\bx_{i}^{k, i - 1})\|^{2} \\
	 		& = H(\bx^{k,i - 1}) - \alpha_{i}^{k}S_{i}(\bx^{k,i - 1}) +
			\frac{(\alpha_{i}^{k})^{2}\beta_{i}}{2}\|\bba_{i}(\bp_{i}^{k} -
			\bx_{i}^{k})\|^{2},
		\end{align*}
		where the last inequality follows from the fact that $\bx_{i}^{k} = \bx_{i}^{k, i - 1}
		$ (see \eqref{BlockUpdate}). This proves the desired result.
	\end{proof}
We also need an explicit expression of the distance between two consecutive iterates generated by the CBCG method. The following lemma also holds for any choice of stepsize.
	\begin{lemma} \label{L:IterateGap}
		Let $\Seq{\bx}{k}$ be the sequence generated by the CBCG method. Then, for any $k \in \nn$, we have
		\begin{equation*}
			\|\bx^{k + 1} - \bx^{k}\|^{2} = \sum_{j = 1}^{N} (\alpha_{j}^{k})^{2}
			\|\bp_{j}^{k} - \bx_{j}^{k,j - 1}\|^{2}.
		\end{equation*}
		Furthermore, for any $i = 0 , 1 , \ldots , N$, we have
		\begin{equation*}
			\|\bx^{k,i}	- \bx^{k}\|^{2} \leq \|\bx^{k + 1} - \bx^{k}\|^{2}.
		\end{equation*}
	\end{lemma}
	\begin{proof}
		For any fixed $k \in \nn$, we have from the definition of the iterates of the CBCG
		algorithm (see \eqref{CBCG:2}) and \eqref{BlockUpdate} that
		\begin{equation*}
			\|\bx^{k + 1} - \bx^{k}\|^{2} = \sum_{j = 1}^{N} \|\bx_{j}^{k,N} - \bx_{j}
			^{k}\|^{2} = \sum_{j = 1}^{N} (\alpha_{j}^{k})^{2}\|\bp_{j}^{k} -
			\bx_{j}^{k,j - 1}\|^{2},
		\end{equation*}
		which proves the first statement. Furthermore, since for any $i = 0 , 1 , \ldots , N$,
		$\bx_{j}^{k,i} = \bx_{j}^{k}$ for all $j > i$ and $\bx_{j}^{k,i} =\bx^{k,N}$ for all $j
		\leq i$, we have
		\begin{align*}
			\|\bx^{k,i} - \bx^{k}\|^{2} & = \sum_{j = 1}^{N} \|\bx_{j}^{k,i} - \bx_{j}
			^{k}\|^{2} = \sum_{j = 1}^{i} \|\bx_{j}^{k,i} - \bx_{j}^{k}\|^{2} = \sum_{j = 1}
			^{i} \|\bx_{j}^{k,N} - \bx_{j}^{k}\|^{2} \nonumber \\
			& \leq \sum_{j = 1}^{N} \|\bx_{j}^{k,N} - \bx_{j}^{k}\|^{2} = \|\bx^{k + 1}
			- \bx^{k}\|^{2},
		\end{align*}
		which proves the second statement and the proof is complete.
	\end{proof}
	
\subsection{Analysis for the Predefined Stepsize Strategy} \label{SSec:Pred}
	In this section, we analyze the CBCG algorithm with the predefined stepsize given by
	\begin{equation} \label{pre-defined}
		\alpha_{i}^{k} = \alpha^{k} \equiv \frac{2}{k + 2}.
	\end{equation}
	\ed{Note that this stepsize rule is completely blind to disparities between blocks and therefore does not allow to counterbalance them in a way that will increase the convergence speed.} The main step towards the proof of convergence rate in this case is recorded in the
	following lemma.
	\begin{lemma} \label{L:PreMain}
		Let $\Seq{\bx}{k}$ be the sequence generated by the CBCG method with the
		stepsize given in \eqref{pre-defined}. Then, for any $k \geq 0$, we have
		\begin{equation*}
			H(\bx^{k + 1}) \leq H(\bx^{k}) - \alpha^{k}S(\bx^{k}
			) + \frac{C_1}{2}(\alpha^{k})^{2},
		\end{equation*}
		where
		\begin{equation} \label{L:PreMain:1}
			C_1 = \sum_{i = 1}^{N} \beta_{i}\norm{\bba_{i}}^{2}D_{i}^{2} + 2L_{F}D\norm{\bba}
			\sum_{i = 1}^{N} D_{i}\norm{\bba_{i}}.
		\end{equation}
	\end{lemma}
	\begin{proof}
		For any $k \geq 0$ and each $i \in \left\{ 1 , 2 , \ldots , N \right\}$, we have from
		Lemma \ref{L:Descent} that
		\begin{align}
			H(\bx^{k,i}) & \leq H(\bx^{k,i - 1}) - \alpha^{k}S_{i}
			(\bx^{k,i - 1}) + \frac{(\alpha^{k})^{2}\beta_{i}}{2}
			\|\bba_{i}(\bp_{i}^{k} - \bx_{i}^{k})\|^{2} \nonumber \\
	 		& \leq H(\bx^{k,i - 1}) - \alpha^{k}S_{i}(\bx^{k,i - 1}) +
	 		\frac{(\alpha^{k})^{2}\beta_{i}\|\bba_{i}\|^{2}D_{i}^{2}}{2}. \label{L:PreMain:2}
		\end{align}
		Summing \eqref{L:PreMain:2} for $i = 1 , 2 , \ldots , N$ yields
		\begin{align}
			H(\bx^{k + 1}) & = H(\bx^{k,N}) \nonumber\\
			& \leq H(\bx^{k,0}) - \alpha^{k}\sum_{i = 1}^{N} S_{i}(\bx^{k,i -
			1}) + \frac{(\alpha^{k})^{2}}{2}\sum_{i = 1}^{N} (\beta_{i}
			\|\bba_{i}\|^{2}D_{i}^{2}) \nonumber \\
			& = H(\bx^{k}) - \alpha^{k}\sum_{i = 1}^{N} S_{i}(\bx^{k}) +
			\frac{(\alpha^{k})^{2}}{2}\sum_{i = 1}^{N} (\beta_{i}\|\bba_{i}\|
			^{2}D_{i}^{2}) + \alpha^{k}\sum_{i = 1}^{N} (S_{i}(\bx^{k})
			- S_{i}(\bx^{k,i - 1})) \nonumber\\
	 		& = H(\bx^{k}) - \alpha^{k}S(\bx^{k}) +
	 		\frac{(\alpha^{k})^{2}}{2}\sum_{i = 1}^{N} (\beta_{i}\|\bba_{i}\|
			^{2}D_{i}^{2}) + \alpha^{k}\sum_{i = 1}^{N} (S_{i}(\bx^{k})
			- S_{i}(\bx^{k,i - 1})), \label{L:PreMain:3}
		\end{align}
		where the last equality follows from \eqref{sum_S}. Using
		Lemma \ref{L:LipSi}, and the fact that $\bx_{i}^{k,i - 1} = \bx_{i}^{k}$  (see  \eqref{BlockUpdate}) gives, for any $i = 1 , 2 , \ldots , N$, that
		\begin{equation} \label{L:PreMain:4}
			S_{i}(\bx^{k}) - S_{i}(\bx^{k,i - 1}) \leq L_{F}D_{i}
			\norm{\bba_{i}} \cdot \|\bba(\bx^{k} - \bx^{k,i - 1})\|.
		\end{equation}
		Combining \eqref{L:PreMain:3} and \eqref{L:PreMain:4} yields
		\begin{equation} \label{L:PreMain:5}
			H(\bx^{k + 1}) \leq H(\bx^{k}) - \alpha^{k}S(\bx^{k}
			) + \frac{(\alpha^{k})^{2}}{2}\sum_{i = 1}^{N} (\beta_{i}
			\|\bba_{i}\|^{2}D_{i}^{2}) + \alpha^{k}L_{F}\sum_{i = 1}^{N} D_{i}
			\|\bba_{i}\| \cdot \|\bba(\bx^{k} - \bx^{k,i - 1})\|.
		\end{equation}
		From Lemma \ref{L:IterateGap}, for any $i = 1 , 2 , \ldots , N$, we have
		\begin{align}
			\|\bba(\bx^{k} - \bx^{k,i - 1})\|^{2} & \leq \|\bba\|^{2} \cdot
			\|\bx^{k} - \bx^{k,i - 1}\|^{2} \leq \|\bba\|^{2} \cdot \|\bx^{k} - \bx^{k +
			1}\|^{2} \nonumber\\
		    & = (\alpha^{k})^{2}\|\bba\|^{2}\sum_{j = 1}^{N} \|\bp_{j}^{k} -
		    \bx_{j}^{k}\|^{2} \leq (\alpha^{k})^{2}\|\bba\|^{2}\sum_{j = 1}^{N} D_{j}
		    ^{2} \nonumber \\
	  		& = (\alpha^{k})^{2}\|\bba\|^{2}D^{2}. \label{L:PreMain:6}
		\end{align}
		Combining \eqref{L:PreMain:5} and \eqref{L:PreMain:6} yields
		\begin{equation*}
			H(\bx^{k + 1}) \leq H(\bx^{k}) - \alpha^{k}S(\bx^{k}
			) + \frac{(\alpha^{k})^{2}}{2}\left [\sum_{i = 1}^{N} \beta_{i}
			\|\bba_{i}\|^{2}D_{i}^{2} + 2L_{F}D\|\bba\| \sum_{i = 1}^{N} D_{i}\|\bba_i\|
			\right ],
		\end{equation*}
		which proves the desired result.
	\end{proof}
	We now recall the following technical lemma on sequences of non-negative numbers (\cf
	\cite[Lemma 1]{B2012}).
	\begin{lemma} \label{L:Technical1}
		Let $\seq{a}{k}$ and $\seq{b}{k}$ be two sequences satisfying, for any $k \geq 0$, that
		\begin{equation} \label{L:Technical1:1}
			a_{k + 1} \leq a_{k} - t_{k}b_{k} + \frac{C}{2}t_{k}^{2},
		\end{equation}
	 	where $0 \leq a_{k} \leq b_{k}$, $t_{k} = 2/(k + 2)$ and $C > 0$, then
		\begin{itemize}
			\item[(a)] $a_{k} \leq 2C/\left(k + 1\right)$;
			\item[(b)] For any $n \geq 1$, we have
			\begin{equation*}
				\min_{k = \lfloor n/2 \rfloor, \ldots , n} b_{k} \leq \frac{8C}{n}.
			\end{equation*}
		\end{itemize}
	\end{lemma}
 	We are now ready to establish the sublinear rate of convergence of the CBCG algorithm with
 	predefined stepsize.
 	\begin{theorem}[Sublinear rate for CBCG with predefined stepsize] \label{T:RatePre}
	 	Let $\Seq{\bx}{k}$ be the sequence generated by the CBCG method with the predefined
		stepsize strategy given in \eqref{pre-defined}.  Then, for any $k \geq 0$, we have
		\begin{equation} \label{T:RatePre:1}
			H(\bx^{k}) - H^{\ast} \leq \frac{2C_1}{k + 1},
		\end{equation}
		and, for any $n \geq 1$, we have
		\begin{equation*}
			\min_{k = \lfloor n/2 \rfloor, \ldots , n} S(\bx^{k}) \leq \frac{8C_1}{n},
		\end{equation*}
  		where
		\begin{equation*}
			C_1 = \sum_{i = 1}^{N} \beta_{i}\norm{A_{i}}^{2}D_{i}^{2} + 2L_{F}D\norm{A} \sum_{i =
			1}^{N} D_{i}\norm{A_{i}}.
		\end{equation*}
	\end{theorem}
	\begin{proof}
		From Lemma \ref{L:PreMain} we have
		\begin{equation*}
			H(\bx^{k + 1}) - H^{\ast} \leq H(\bx^{k}) - H^{\ast} -
			\alpha^{k}S(\bx^{k}) + \frac{C_1}{2}(\alpha^{k})^{2}.
		\end{equation*}
		In addition, by Lemma \ref{L:UppBoundS}, we have that $S(\bx^{k}) \geq
		H(\bx^{k}) - H^{\ast}$. We can therefore invoke Lemma \ref{L:Technical1}
		with $a_{k} = H(\bx^{k}) - H^{\ast}$ and $b_{k} = S(\bx^{k})$,
		and the desired result is established.
	\end{proof}
	
\subsection{Analysis for the Adaptive Stepsize Strategy} \label{SSec:Adap}
	We now turn to the analysis of CBCG with adaptive stepsize. The main insight is to choose
	a stepsize that minimizes the \ed{quadratic} upper bound given in Lemma \ref{L:Descent}. In this setting,
	the adaptive stepsize is defined as follows:
	\begin{align}
		\alpha_{i}^{k} & = \argmin_{\alpha \in \left[0 , 1\right]} \left\{ -\alpha S_{i}
		(\bx^{k,i - 1}) + \alpha^{2} \cdot \frac{\beta_{i}}{2}\|\bba_{i}(
		\bp_{i}^{k} - \bx_{i}^{k})\|^{2} \right\} \nonumber \\
		& = 	\min \left\{ \frac{S_{i}(\bx^{k,i - 1})}{\beta_{i}\|\bba_{i}
		(\bp_{i}^{k} - \bx_{i}^{k})\|^{2}} , 1 \right\}. \label{adaptive}
	\end{align}
	Note that by Lemma \ref{L:Descent}, this choice of stepsize leads to a decrease of
	objective value at each step. The analysis employed for the predefined stepsize does not
	seem to be easily adjusted to the adaptive stepsize rule. Thus, a different analysis is
	developed in this section. We begin with the following technical result.
	\begin{lemma} \label{L:DescentAdap}
		Let $\Seq{\bx}{k}$ be the sequence generated by the CBCG method with the adaptive
		stepsize strategy given in \eqref{adaptive}.  Then, for any $k \geq 0$ and $i \in
		\left\{ 1 , 2 , \ldots , N \right\}$, we have
		\begin{equation*}
			H(\bx^{k,i - 1}) - H(\bx^{k,i}) \geq \frac{\alpha_{i}^{k}}{2}
			S_{i}(\bx^{k,i - 1}) \geq \frac{(\alpha_{i}^{k})^{2}
			\beta_{i}}{2}\|\bba_{i}(\bp_{i}^{k} - \bx_{i}^{k})\|^{2}.
		\end{equation*}
	\end{lemma}
	\begin{proof}
		We split the proof into two cases. First, if
		\begin{equation} \label{L:DescentAdap:1}  		
			\frac{S_{i}(\bx^{k,i - 1})}{\beta_{i}\|\bba_{i}(\bp_{i}^{k} -
			\bx_{i}^{k})\|^{2}} \geq 1,
		\end{equation}
		then $\alpha_{i}^{k} = 1$, and by Lemma \ref{L:Descent} we have
		\begin{align}
			H(\bx^{k,i - 1}) - H(\bx^{k,i}) & \geq \alpha_{i}^{k}
			S_{i}(\bx^{k,i - 1}) - \frac{(\alpha_{i}^{k})^{2}\beta_{i}}
			{2}\|\bba_{i}(\bp_{i}^{k} - \bx_{i}^{k})\|^{2} \nonumber \\
&= S_i(\bx^{k,i-1})-\frac{\beta_i}{2}\|\bba_i (\bx^k_i-\bp^k_i)\|^2\\
&\geq \frac{\alpha^k_i}{2}S_i(\bx^{k,i-1})\geq \frac{(\alpha^k_i)^2\beta_i}{2}\|\bba_i (\bx^k_i-\bp^k_i)\|^2,\label{L:DescentAdap:2}
		\end{align}
where the last two inequalities follows from (\ref{L:DescentAdap:1}) and the fact that $\alpha^k_i=1$.  In the second case, when $S_{i}(\bx^{k,i - 1}) <
		\beta_{i}\|\bba_{i}(\bp_{i}^{k} - \bx_{i}^{k})\|^{2}$, we have that
		\begin{equation*}
			\alpha_{i}^{k} = \frac{S_{i}(\bx^{k,i - 1})}{\beta_{i}\|\bba_{i}
			(\bp_{i}^{k} - \bx_{i}^{k})\|^{2}}.
		\end{equation*}
		Thus,
		\begin{align}
			H(\bx^{k,i - 1}) - H(\bx^{k,i}) & \geq \alpha_{i}^{k}
			S_{i}(\bx^{k,i - 1}) - \frac{(\alpha_{i}^{k})^{2}\beta_{i}}
			{2}\|\bba_{i}(\bp_{i}^{k} - \bx_{i}^{k})\|^{2} \nonumber \\
			& = \frac{S_{i}(\bx^{k,i - 1})^{2}}{2\beta_{i}\|\bba_{i}(\bp_{i}
			^{k} - \bx_{i}^{k})\|^{2}} \nonumber \\
			& = \frac{\alpha_i^k}{2}S_i(\bx^{k,i-1})=\frac{(\alpha_{i}^{k})^{2}\beta_{i}}{2}\|\bba_{i}(\bp_{i}^{k}
			- \bx_{i}^{k})\|^{2}. \label{L:DescentAdap:3}
		\end{align}		
		The result now follows by combining \eqref{L:DescentAdap:2} and
		\eqref{L:DescentAdap:3}.
	\end{proof}
	We can now prove the following important result.
	\begin{lemma} \label{L:DescentAdap2}
		Let $\Seq{\bx}{k}$ be the sequence generated by the CBCG method with the adaptive
		stepsize strategy given in \eqref{adaptive}. Then, for any $k \geq 0$ and $i \in
		\left\{ 1 , 2 , \ldots , N \right\}$, we have
		\begin{equation*}
			H(\bx^{k,i - 1}) - H(\bx^{k,i}) \geq \frac{S_{i}
			(\bx^{k,i - 1})^{2}}{2\max \left\{ \beta_{i}\|\bba_i\|^{2} D_{i}^{2} ,
			K_{i} \right\}},
		\end{equation*}
		where, for each $i = 1 , 2 , \ldots , N$,
		\begin{align}
			\label{eq:defKi}
			M_{i} &= \max_{\bx \in X} \norm{\nabla_{i} f(\bx)},\\
			K_{i} &= (M_{i} + l_{i})D_{i}
		\end{align}
	\end{lemma}
	\begin{proof}
		We again split the proof into two cases. First, if $\alpha_{i}^{k} = 1$, then by Lemma
		\ref{L:DescentAdap} we have
		\begin{equation} \label{L:DescentAdap2:1}
			H(\bx^{k,i - 1}) - H(\bx^{k,i}) \geq \frac{1}{2}S_{i}
			(\bx^{k,i - 1}) \geq \frac{S_{i}(\bx^{k,i - 1})^{2}}{2K_{i}},
		\end{equation}
		where the last inequality follows from the fact that for any $\bx \in X$:
		\begin{align}
			S_{i}(\bx) & = \act{\nabla_{i} f(\bx) , \bx_{i} - \bp_{i}
			(\bx)} + g_{i}(\bx_{i}) - g_{i}(\bp_{i}(\bx)
			) \nonumber\\
			& \leq \norm{\nabla_{i} f(\bx)} \cdot \norm{\bx_{i} - \bp_{i}(\bx
			)} + l_{i}\norm{\bx_{i} - \bp_{i}(\bx)} \nonumber \\
			& \leq (M_{i} + l_{i})D_{i}. \label{L:DescentAdap2:2}
		\end{align}
		In the second case, when $\alpha_{i}^{k} < 1$, we have that	$\alpha_{i}^{k} = \frac{S_{i}(\bx^{k,i - 1})}{\beta_{i}\|\bba_{i}
			(\bp_{i}^{k} - \bx_{i}^{k})\|^{2}}$, and by Lemma \ref{L:DescentAdap} we can write
		\begin{equation} \label{L:DescentAdap2:3}
			H(\bx^{k,i - 1}) - H(\bx^{k,i}) \geq \frac{S_{i}
			(\bx^{k,i - 1})^{2}}{2\beta_{i}\|\bba_{i}(\bp_{i}^{k} - \bx_{i}
			^{k})\|^{2}} \geq \frac{S_{i}(\bx^{k,i - 1})^{2}}{2\beta_{i}
			\|\bba_i\|^{2}D_{i}^{2}},
		\end{equation}
		The result now follows by combining \eqref{L:DescentAdap2:2} and
		\eqref{L:DescentAdap2:3}.
	\end{proof}
	We will now prove an additional technical lemma that establishes a ``sufficient decrease"
	property of the CBCG method with adaptive stepsize.
	\begin{lemma} \label{L:Suff}
		Let $\Seq{\bx}{k}$ be the sequence generated by the CBCG method with the adaptive
		stepsize strategy given in \eqref{adaptive}. Then, for any $k \geq 0$ and $i \in
		\left\{ 1 , 2 , \ldots , N \right\}$, we have
		\begin{equation*}
			H(\bx^{k}) - H(\bx^{k + 1}) \geq \frac{\beta_{min}}{2N }
			\|\bba (\bx^{k} - \bx^{k,i-1})\|^{2},
		\end{equation*}
		where $\beta_{\min}$ is given by \eqref{def_beta_min}.
	\end{lemma}
	\begin{proof}
		By Lemma \ref{L:DescentAdap} we have
		\begin{equation} \label{L:Suff:1}
			H(\bx^{k,i - 1}) - H(\bx^{k,i}) \geq \frac{(\alpha_{i}
			^{k})^{2}\beta_{i}}{2}\|\bba_{i}(\bp_{i}^{k} - \bx_{i}^{k})\|^{2}.
		\end{equation}
		Now, for any $i \in \left\{ 1 , 2 , \ldots , N \right\}$, we can write
 		\begin{align*}
			\|\bba(\bx^{k} - \bx^{k,i-1})\|^{2} & = \left \| \sum_{j = 1}^{N} \bba_{j}
			(\bx_{j}^{k} - \bx_{j}^{k,i-1})\right \|^{2} \\
& \leq N\sum_{j = 1}^{N} \|\bba_{j}
			(\bx_{j}^{k} - \bx_{j}^{k,i-1})\|^{2} \\
			& = N\sum_{j = 1}^{i - 1} \|\bba_{j}(\bx_{j}^{k} - \bx_{j}^{k,i - 1})\|
			^{2} \\
&= N\sum_{j = 1}^{i - 1} (\alpha_{i}^{k})^{2}\|\bba_{j}
			(\bx_{j}^{k} - \bp_{j}^{k})\|^{2} \\
			& \leq \frac{2N}{\beta_{min}}\sum_{j = 1}^{i - 1} (H(\bx^{k,j - 1}) -
			H(\bx^{k,j})) \\
			& = \frac{2N}{\beta_{min}}(H(\bx^{k}) - H(\bx^{k,i - 1})
			) \\
			& \leq \frac{2N}{\beta_{min}}(H(\bx^{k}) - H(\bx^{k + 1}
			)),
 		\end{align*}
		where the first inequality follows from the fact that for any $N$ vectors $\bu_{1} ,
		\bu_{2} , \ldots , \bu_{N}$, the inequality $\norm{\sum_{j = 1}^{N} \bu_{j}}^2 \leq N\sum_{j
		= 1}^{N} \norm{\bu_{j}}^{2}$ holds; the second inequality follows from \eqref{L:Suff:1}, and
		the last inequality follows from Lemma \ref{L:DescentAdap}, which shows that the
		sequence of function values is non-increasing.
	\end{proof}
\begin{remark}
	\label{rem:improvedConstant}
	The bound in Lemma \ref{L:Suff} can be improved in some situations where additional information on the structure of $\bba$ is available. For example, if the column space of each $\bba_i$ span orthogonal spaces, that is $\bba_i^T \bba_j = \bo$ for any $1 \leq i < j \leq N$, then the factor $N$ can be avoided.
\end{remark}
	The next lemma constitutes the crucial step towards the establishment of a sublinear convergence rate of
	the CBCG method with adaptive stepsize.
	\begin{lemma} \label{L:MainAda}
		Let $\Seq{\bx}{k}$ be the sequence generated by the CBCG method with the adaptive
		stepsize strategy given in \eqref{adaptive}. Then, for any $k \geq 0$, we have
		\begin{equation*}
			H(\bx^{k}) - H(\bx^{k + 1}) \geq \frac{1}{NC_2}S(\bx^{k}
			)^{2},
		\end{equation*}
		where
		\begin{equation} \label{L:MainAda:1}
			C_2 = 4\left [\max_{i = 1 , 2 , \ldots , N}\left\{\max \left\{ \beta_{i}\norm{A_{i}}^{2}
			D_{i}^{2} , K_{i} \right\} \right\} + \frac{NL_{F}^{2}D^{2}\max_{i = 1 , 2 , \ldots , N}
			\norm{A_{i}}^{2}}{\beta_{\min}}\right ],
		\end{equation}
		and $\beta_{\min}$ is defined in \eqref{def_beta_min} while $K_i$ is defined in \eqref{eq:defKi}, for $i = 1, 2, \ldots, N$.
	\end{lemma}
	\begin{proof}
		For any $i \in \left\{ i = 1 , 2 , \ldots , N \right\}$ we have that
		\begin{align}
			S_{i}(\bx^{k})^{2} & = (S_{i}(\bx^{k,i - 1}) + S_{i}
			(\bx^{k}) - S_{i}(\bx^{k,i - 1}))^{2} \nonumber \\
			& \leq 2S_{i}(\bx^{k,i - 1})^{2} + 2(S_{i}(\bx^{k}) -
			S_{i}(\bx^{k,i - 1}))^{2} \nonumber \\
			& \leq 2S_{i}(\bx^{k,i - 1})^{2} + 2L_{F}^{2}D_{i}^{2}\|\bba_{i}\|^{2}
			\cdot \|\bba(\bx^{k} - \bx^{k,i - 1})\|^{2} \nonumber \\
			& \leq 2S_{i}(\bx^{k,i - 1})^{2} + \frac{4N L_{F}^{2}D_{i}^{2}\|\bba_{i}
			\|^{2}}{\beta_{\min}}(H(\bx^{k}) - H(\bx^{k + 1})
			), \label{L:MainAda:2}
		\end{align}
		where the second inequality follows from Lemma \ref{L:LipSi} and the last inequality
		follows from Lemma \ref{L:Suff}. Invoking Lemma \ref{L:DescentAdap2}, we obtain from
		\eqref{L:MainAda:2} that
		\begin{align}
			S_{i}(\bx^{k})^{2} & \leq 2S_{i}(\bx^{k,i - 1})^{2} +
			\frac{4N L_{F}^{2}D_{i}^{2}\|\bba_i\|^{2}}{\beta_{\min}}(H(\bx^{k})
			- H(\bx^{k + 1})
			) \nonumber \\
			& \leq 4\max\left\{ \beta_{i}\|\bba_i\|^{2}D_{i}^{2} , K_{i} \right\}(H
			(\bx^{k,i - 1}) - H(\bx^{k,i})) + \frac{4NL_{F}^{2}D_{i}
			^{2}\|\bba_i\|^{2}}{\beta_{\min}}(H(\bx^{k}) - H(\bx^{k + 1}
			)). \label{L:MainAda:3}
		\end{align}
		Summing \eqref{L:MainAda:3} for $i = 1 , 2 , \ldots , N$ yields
 		\begin{align*}
			\sum_{i = 1}^{N} S_{i}(\bx^{k})^{2} & \leq 4\sum_{i = 1}^{N} \max\left\{
			\beta_{i}\|\bba_i\|^{2}D_{i}^{2} , K_{i} \right\}(H(\bx^{k,i - 1}
			) - H(\bx^{k,i})) \\
			& + \frac{4N L_{F}^{2}}{\beta_{\min}}\sum_{i = 1}^{N} D_{i}^{2}\|\bba_i\|^{2}
			(H(\bx^{k}) - H(\bx^{k + 1})) \\
			& \leq 4\max_{i = 1 , 2 , \ldots , N} \max\left\{ \beta_{i}\|\bba_i\|^{2}D_{i}
			^{2} , K_{i} \right\} \sum_{i = 1}^{N} (H(\bx^{k,i - 1}) - H
			(\bx^{k,i})) \\
			& + \frac{4NL_{F}^{2}D^{2}\max_{i = 1 , 2 , \ldots , N} \|\bba_i\|^{2}}
			{\beta_{\min}}(H(\bx^{k}) - H(\bx^{k + 1})) \\
			& = C_2(H(\bx^{k}) - H(\bx^{k + 1})).
		\end{align*}
		Finally, using \eqref{sum_S} we have
 		\begin{equation*}
			S(\bx^{k})^{2} = \left [\sum_{i = 1}^{N} S_{i}(\bx^{k})
			\right ]^{2} \leq N\sum_{i = 1}^{N} S_{i}(\bx^{k})^{2}\leq NC_2(H
			(\bx^{k}) - H(\bx^{k + 1})),
		\end{equation*}
		which proves the desired result.
	\end{proof}
	By combining Lemma \ref{L:UppBoundS} with Lemma \ref{L:MainAda}, we obtain the following
	corollary.
	\begin{corollary} \label{C:Suff}
		Let $\Seq{\bx}{k}$ be the sequence generated by the CBCG method with the adaptive
		stepsize strategy given in \eqref{adaptive}. Then, for any $k \geq 0$, we have
		\begin{equation*}
			H(\bx^{k}) - H(\bx^{k + 1}) \geq \frac{1}{NC_2}(H
			(\bx^{k}) - H^*)^2,
		\end{equation*}
		where $C_2$ is given in \eqref{L:MainAda:1}.
	\end{corollary}
	In order to get rate of convergence result in this case we will need the following
	technical lemma (\cf \cite[Lemma 3.5]{BTet2013}).
	\begin{lemma} \label{L:Technical2}
 		Let $\seq{a}{k}$ be a sequence of non-negative real numbers satisfying, for any $k \in
 		\nn$, the following property
		\begin{equation} \label{L:Technical2:1}
			a_{k} - a_{k + 1} \geq \gamma a_{k}^{2},
		\end{equation}
		and $a_{0} \leq 1/(\gamma m)$ for some positive $\gamma$ and $m$. Then, for
		any $k \in \nn$, we have that
		\begin{equation*}
			a_{k} \leq \frac{1}{\gamma} \cdot \frac{1}{m + k}.
		\end{equation*}
	\end{lemma}
Combining Lemma \ref{L:Technical2} along with Corollary  \ref{C:Suff} and Lemma \ref{L:UppBoundS}, establishes the
sublinear rate of convergence of the CBCG method with adaptive stepsize.

	\begin{theorem}[Sublinear rate for CBCG with adaptive stepsize] \label{T:RateAdap}
		Let $\Seq{\bx}{k}$ be the sequence generated by the CBCG method with the adaptive
		stepsize strategy given in \eqref{adaptive}.  Then for any $k \geq 0$ we have
		\begin{equation} \label{T:RateAdap:1}
			H(\bx^{k}) - H^{\ast} \leq \frac{NC_2}{k + 4},
		\end{equation}
		and
		\begin{equation*}
			\min_{n = 0 , 1 , \ldots , k} S(\bx^{n}) \leq \frac{2NC_2}{k + 4},
		\end{equation*}
		where
		\begin{equation*}
			C_2 = 4 \left [\max_{i = 1 , 2 , \ldots , N} \left\{\max \left\{ \beta_{i}\|\bba_i\|^{2}
			D_{i}^{2} , K_{i} \right\}  \right\} + \frac{NL_{F}^{2}D^{2}\max_{i = 1 , 2 , \ldots , N}
			\|\bba_i\|^{2}}{\beta_{\min}}\right ],
		\end{equation*}
		$\beta_{\min}$ is defined in \eqref{def_beta_min} and $K_i$ is defined in (\ref{eq:defKi}), for $i = 1, 2, \ldots, N$.
	\end{theorem}
	\begin{proof}
		Denote $a_{k} = H(\bx^{k}) - H^{\ast}$. Thanks to Corollary \ref{C:Suff} we
		get that \eqref{L:Technical2:1} holds true with $\gamma = 1/(NC_2)$. In
		addition, from Lemma \ref{L:UppBoundS} and \eqref{L:DescentAdap2:2}, we have
		\begin{equation*}
			a_{0} = H(\bx^{0}) - H^{\ast} \leq S(\bx^{0}) = \sum_{i = 1}
			^{N} S_{i}(\bx^{0}) \leq \sum_{i = 1}^{N} K_{i} \leq N\max_{i = 1 , 2 ,
			\ldots , N} K_{i} \leq \frac{NC_2}{4}.
		\end{equation*}
		Hence, $a_{0} \leq 1/(4\gamma)$. Picking $m = 4$, we conclude from Lemma
		\ref{L:Technical2} that \eqref{T:RateAdap:1} holds. To find the bound on the optimality
		measure $S(\cdot)$, from Lemma \ref{L:MainAda}, we have for any $n \geq 0$
		\begin{equation*}
			H(\bx^{n}) - H(\bx^{n + 1}) \geq \gamma S(\bx^{n}
			)^{2}.
		\end{equation*}	
		For any $k_{0} \geq 0$, summing the latter inequality for $n = k_{0} , k_{0} + 1 ,
		\ldots , 2k_{0}$, we obtain that
		\begin{equation*}
			\gamma \sum_{n = k_{0}}^{2k_{0}} S(\bx^{n})^{2} \leq H(\bx^{k_{0}}
			) - H(\bx^{2k_{0} + 1}) \leq H(\bx^{k_{0}}) - H
			(\bx^{\ast}) \leq \frac{1}{\gamma} \cdot \frac{1}{k_{0} + 4},
		\end{equation*}	
		where the last inequality follows from \eqref{T:RateAdap:1}. Hence,
		\begin{equation} \label{T:RateAdap:2}
			\min_{n = k_{0} , k_{0} + 1 , \ldots , 2k_{0}} S(\bx^{n})^{2} \leq 		
			\frac{1}{\gamma^{2}} \cdot \frac{1}{(k_{0} + 1)(k_{0} + 4)}
			\leq \frac{1}{\gamma^{2}} \cdot \frac{1}{(k_{0} + 2)^{2}},
		\end{equation}
		where we have used the fact that $(k_{0} + 1)(k_{0} + 4) \geq
		(k_{0} + 2)^{2}$. Similarly, for any $k_{0} \geq 0$,
		\begin{equation*}
			\gamma \sum_{n = k_{0}}^{2k_{0} + 1} S(\bx^{n})^{2} \leq H
			(\bx^{k_{0}}) - H(\bx^{2k_{0} + 2}) \leq H(\bx^{k_{0}}
			) - H(\bx^{\ast}) \leq \frac{1}{\gamma} \cdot \frac{1}{k_{0} + 4}.
		\end{equation*}	
		Hence
		\begin{equation} \label{T:RateAdap:3}
			\min_{n = k_{0} , k_{0} + 1 , \ldots , 2k_{0} + 1} S(\bx^{n})^{2} \leq 	
			\frac{1}{\gamma^{2}} \cdot \frac{1}{(k_{0} + 2)(k_{0} + 4)}
			\leq \frac{1}{\gamma^{2}} \cdot \frac{1}{(k_{0} + 5/2)^{2}},
		\end{equation}
		where we have used the fact that $(k_{0} + 2)(k_{0} + 4) \geq
		(k_{0} + 5/2)^{2}$. Notice that $k_{0} + 5/2 = (2k_{0} + 1)/2 +
		2$. Therefore, by combining \eqref{T:RateAdap:2} when $k$ is even and \eqref{T:RateAdap:3} when $k$ is odd, we conclude
		that
		\begin{equation*}
			\min_{n = 0 , 1 , \ldots , k} S(\bx^{n}) \leq \frac{1}{\gamma} \cdot
			\frac{1}{k/2 + 2} \leq \frac{2NC_2}{k + 4},
		\end{equation*}
		which concludes the proof.
	\end{proof}

\subsection{Backtracking Version of CBCG} \label{SSec:Back}
	Computing the adaptive stepsize requires to know the constants $\beta_{i}$, $i = 1 , 2 , \ldots , N$. In practice, these constants may not be known in advance or their known
	upper-approximations may be too loose. A common approach to overcome this problem is
	to use a backtracking scheme in order to estimate the unknown constants that are required
	to ensure convergence of the algorithm. This strategy is also effective in the context of
	CG-type algorithms.

The crucial property for the convergence analysis with adaptive stepsize is the block structured descent condition given in Lemma \ref{L:DescentAdap}. When the constants in Lemma \ref{L:DescentAdap} are not known, a workaround is to enforce this descent condition algorithmically, which leads to the following scheme.
\bigskip

\fcolorbox{black}{Ivory2}{
	\parbox{15cm}{
		{\bf CBCG-B: Cyclic Block Conditional Gradient with Backtracking} \\
		{\bf Initialization.} $\bx^0 \in X$, $\beta_{\rm init} > 0$, $\kappa > 1$, $\xi^{-1}_i =1$, $i = 1, 2, \ldots, N$.\\
		{\bf General Step.}	For $k = 1 , 2 , \ldots$,
		\begin{itemize}
			\item[$\rm{(i)}$] For any $i = 1 , 2 , \ldots , N$, compute
				\begin{align}
					\bp_{i}^{k} \in \argmin_{\bp_{i} \in X_{i}} \left\{ \langle \nabla_{i}
					f(\bx^{k,i - 1}) , \bp_{i} \rangle  + g_{i}(\bp_{i}) \right\},
					\label{eq:CBCG-B1}
				\end{align}
				Find $\xi \in \NN$, which is the smallest integer $\xi \geq \xi^{k-1}_i$ such that
				\begin{align}
					\label{eq:backtrack}
					H\left(\bx^{k,i-1}\right)- H\left(\bx^{k,i-1} + \alpha \bbu_i (\bp^k_{i} -  \bx^{k,i-1}_i)\right) \geq \frac{\alpha}{2} S_i(\bx^{k,i-1}),
				\end{align}
				where $\alpha = \min\left\{ \frac{S_i(\bx^{k,i-1})}{\kappa^{\xi} \beta_{\rm init} ||\bba_i (\bp^k_{i} -  \bx^{k,i-1}_i)||^2}, 1 \right\}$ and set
				\begin{align}
					\label{eq:backtrack2}
					&\xi^k_i = \xi,\quad\beta^k_i = \kappa^{\xi} \beta_{\rm init},\quad\alpha^k_i = \min\left\{ \frac{S_i(\bx^{k, i-1})}{\beta^k_i ||\bba_i (\bp^k_{i} -  \bx^{k,i-1}_i)||^2}, 1 \right\},\\
					&\bx^{k,i} = \bx^{k,i-1} + \alpha^k_i \bbu_i (\bp^k_{i} -  \bx^{k,i-1}_i).
				\end{align}
			\item[$\rm{(ii)}$] Update: $\bx^{k+1} = \bx^{k,N}$.
		\end{itemize}
	}
}
\bigskip

	Under Assumption \ref{AssumptionB}, we have from Lemma \ref{L:DescentAdap}, that
	\eqref{eq:backtrack} holds provided that $\beta_{i}^{k} \geq \beta_{i}$. We therefore have for
	any $k \geq 0$ and any $i = 1 , 2 , \ldots , N$ that
	\begin{equation} \label{eq:backtrackBounded}
		\beta_{init} \leq \beta_{i}^{k} \leq \bar{\beta_{i}} \equiv \max \left\{ \kappa
		\beta_{i} , \beta_{init} \right\}.
	\end{equation}
	The main insight is given in the following lemma which its proof follows the same
	arguments as that of Lemma \ref{L:DescentAdap} using \eqref{eq:backtrack}.
	\begin{lemma} \label{L:DescentAdapBack}
		Let $\Seq{\bx}{k}$ be the sequence generated by the CBCG-B method.  Then, for any $k \geq 0$ and $i \in \left\{ 1 , 2
		, \ldots , N \right\}$, we have
		\begin{equation*}
			H(\bx^{k,i - 1}) - H(\bx^{k,i}) \geq \frac{\alpha_{i}^{k}}{2}
			S_{i}(\bx^{k,i - 1}) \geq \frac{(\alpha_{i}^{k})^{2}
			\beta_{i}^k}{2}\|\bba_{i}(\bp_{i}^{k} - \bx_{i}^{k})\|^{2},
		\end{equation*}
		where
		\begin{equation*}
			\alpha_{i}^{k} = \min\left\{ \frac{S_{i}(\bx^{k,i - 1})}{\beta_{i}^{k}
			\|\bba_{i}(\bp_{i}^{k} - \bx_{i}^{k})\|^{2}} , 1 \right\}.
		\end{equation*}
	\end{lemma}
	Using the bounds in \eqref{eq:backtrackBounded}, we obtain the following rate of
	convergence result for the CBCG algorithm with a backtracking scheme.
	\begin{theorem}[Sublinear rate for CBCG-B] \label{T:RateBack}
		Let $\Seq{\bx}{k}$ be the sequence generated by the CBCG-B method.  Then, for any $k \geq 0$, we have
		\begin{equation*}
			H(\bx^{k}) - H^{\ast} \leq \frac{NC_3}{k + 4},
		\end{equation*}
		and
		\begin{equation*}
			\min_{n = 0 , 1 , \ldots , k} S(\bx^{n}) \leq \frac{2NC_3}{k + 4},
		\end{equation*}
		where
		\begin{equation*}
			C_3 = 4\left [\max_{i = 1 , 2 , \ldots , N} \max \left\{ \bar{\beta_{i}}\|\bba_i\|^{2}D_{i}^{2} , K_{i} \right\} + \frac{NL_{F}^{2}D^{2}\max_{i = 1 , 2 , \ldots ,
			N} \|\bba_i\|^{2}}{\beta_{init}}\right],
		\end{equation*}
		and $\bar{\beta_{i}}$ is given in (\ref{eq:backtrackBounded}) while $K_{i}$ is defined in \eqref{eq:defKi}, for $i = 1, 2, \ldots, N$
	\end{theorem}
	\begin{proof}
		 The arguments of this proof are similar to those in the proof of Theorem
		 \ref{T:RateAdap} where we replace Lemma \ref{L:DescentAdap} with Lemma
		 \ref{L:DescentAdapBack}.
	\end{proof}
\section{\ed{Further Discussions}}
\label{Sec:ML}
\ed{In this section we begin by showing that the  complexity analysis presented in Section \ref{SSec:Adap} holds for the CBCG method with exact line search when the smooth part of the objective function is quadratic. Then, we discuss two issues that are relevant for the implications of our analysis of the CBCG method.}
	
\subsection{Exact Line Search for Quadratic Problems} \label{SSec:Exact}
	We recall another well known stepsize rule --  exact line search strategy (see, for
	example, \cite{DR1967}). This stepsize is defined as follows:
	\begin{equation} \label{exact}
		\alpha_{i}^{k} \in \argmin_{\alpha \in \left[0 , 1\right]} H(\bx^{k,i - 1} +
		\alpha \bbu_{i}(\bp_{i}^{k} - \bx_{i}^{k,i - 1})).
	\end{equation}
The minimization step (\ref{exact}) can incur unnecessary additional computations, unless it can be carried out efficiently. This is the case for quadratic functions for which it reduces to the minimization of a one-dimensional quadratic over a segment. It is therefore tempting to use CBCG with this stepsize rule, but our convergence analysis does not cover it explicitly. However, in the quadratic case, exact line search can be recast as an adaptive stepsize strategy. Indeed, suppose for example that the problem has the following form:
$$ \min \left \{ \frac{1}{2}\left \|\sum_{i=1}^N \bba_i \bx_i \right \|^2+ \sum_{i=1}^n \langle \bb_i,\bx_i \rangle + \ed{\sum_{i=1}^n\tilde{g}_i(\bx_i)}\right\},$$
where $\bba_i \in \real^{n \times n_i}$, $\bb_i \in \real^{n_i}$ and $\tilde{g}_i$ satisfy Assumption \ref{AssumptionA}(i). This problem fits model (\ref{mainproblem}) with $F(\cdot) = \frac{1}{2}\|\cdot\|^2, \bba = (\bba_1,\bba_2,\ldots,\bba_N)$ \ed{and $g_i(\cdot)=\langle \bb_i,\cdot \rangle + \tilde{g}_i(\cdot)$}.
Therefore, we can choose $\beta_{i} = 1$, $i = 1 , 2 , \ldots , N$ and the exact line search strategy is equivalent to our adaptive stepsize strategy given in \eqref{adaptive} since the upper bound of Lemma \ref{L:Descent} holds with equality. Therefore, in the quadratic case, convergence for the exact line search strategy follows from Theorem \ref{T:RateAdap}.

\subsection{\ed{Random permutations}}
\ed{All the arguments presented in Section \ref{Sec:Analysis} remain valid if the order of the blocks is not fixed for all iterations. The only important element is that all blocks are visited once at each iteration, but the order could change, in an arbitrary way, at each iteration. In particular, the arguments are valid when the order of blocks  is picked as a random permutation at the beginning of each iteration. Therefore, all the convergence rate estimates presented in Section \ref{Sec:Analysis} still hold true when using random permutations.

The practical motivation for considering random permutations in the order of the blocks is that it may be beneficial in practical settings \cite{shalev2013stochastic}. Since the derived theoretical efficiency estimates apply to any rule for choosing the order of the blocks at the beginning of each iteration (deterministic or random), they actually do not explain the potential performance differences between different variants, such as purely cyclic versus random permutations.}

\subsection{\ed{Implications for Coordinate Descent in Machine Learning}}
\ed{In the case of blocks of size one (i.e., $N=n$), CBCG method bears a lot of resemblance to coordinate descent type methods. For example, in dimension one, a conditional gradient step with exact line search is the same as exact minimization. Therefore, CBCG with blocks of size one and exact line search is equivalent to coordinate descent with exact minimization. As we have seen in the previous section, our convergence analysis holds for this setting in the case of quadratic objectives. This setting has some interesting applications.

The $\ell_2$ regularized empirical risk minimization is considered in \cite{shalev2013stochastic}, where  the authors analyze the stochastic coordinate descent method with exact minimization in the dual -- a method referred to as ``Stochastic Dual Coordinate Ascent" (SDCA). The dual problem can be written as follows (using our notations):
\begin{align}
	\label{eq:SDCA}
	\min_{\bx}\left\{\frac{1}{n} \sum_{i=1}^n g_i(\bx_i) + \frac{\lambda}{2} \left\|\frac{1}{\lambda n} \sum_{i=1}^n \bba_i \bx_i \right\|^2  \right\},
\end{align}
where $\lambda>0$ is a parameter, and for any $i=1,2,\ldots,n$, $\bx_i \in \real$, $\bba_i \in \real^n$ and $g_i$ is a closed, convex and proper function with a bounded domain. This is exactly the setting of Section \ref{SSec:Exact} and therefore Theorem \ref{T:RateAdap} gives explicit deterministic rate estimates of the duality gap for the cyclic and random permutation coordinate descent variants applied to problem (\ref{eq:SDCA}) when $g_1,g_2,\ldots,g_n$ satisfy Assumption \ref{AssumptionA}. It can be checked that in the setting of problem (\ref{eq:SDCA}) (with $\lambda = 1$ for simplicity), we get in Theorem \ref{T:RateAdap} that as $n$ grows, the dominant term in the rate estimate (\ref{T:RateAdap:1}) is of the form $4 D^2 \max_i\|\bba_i\|^2 / (k+4)$. 
 The quantity $D^2$ remains proportional to $n$ (see identity (\ref{sumd})) and the rate estimate of (\ref{T:RateAdap:1}) still suffers from a multiplicative dependence in $n$, since each iteration requires an effective pass through the $n$ coordinates. Therefore, these results remain mostly of theoretical interest in the context of coordinate descent in machine learning where the focus is on large $n$ and there is space for improvement in order to theoretically grasp the practical performances of these types of methods.
}

\section{Numerical Experiments} \label{Sec:Numerics}
	The experiments presented in this section correspond to situations where for each $i$, the nonsmooth
	function $g_{i}$, is taken to be the indicator of the set $X_{i}
	\subseteq \real^{n_{i}}$, with the eventual addition of a linear term. We therefore recover the smooth
	constrained optimization model of the traditional conditional gradient method (see also Remark
	\ref{rem:tradiCG}).\delete{ Furthermore, for the cyclic selection rule, in all the experiments, we
	use the ``random permutation'' approach which consists in randomly changing the order of
	the blocks at the beginning of each iteration. The convergence analyses described so far do not depend on the order of the blocks and are still valid when it changes at each iteration. Therefore, all the results of Section \ref{Sec:Analysis}  hold deterministically for this ``random permutation'' rule which we do not differentiate from the cyclic update rule, and hence the corresponding algorithms bear the same name in this section.} We begin with a modeling remark in Section \ref{SSec:Exact}. Numerical results are given in Section \ref{SSec:Numerics-Synthetic} for synthetic examples and in Section \ref{SSec:Numerics-SVM} for the \ed{structured} SVM training problem.

\subsection{Synthetic Problems} \label{SSec:Numerics-Synthetic}
	We generate artificial convex quadratic problems with box constraints in $\real^{100}$. We 	
	consider problems of the following form
	\begin{equation} \label{RandomPB}
		\min_{\norm{\bx}_{\infty} \leq 1} \frac{1}{2}(\bx - \by)^{T}\bbq(\bx - \by
		),
	\end{equation}
	where $\norm{\cdot}_{\infty}$ denotes the $\ell_{\infty}$ norm, $\by \in \real^{100}$ and
	$\bbq \in \real^{100 \times 100}$ are given. The problem has a natural coordinate-wise
	structure which allows to apply the CBCG algorithm where each block consists of a single
	coordinate. Indeed, by setting $f(\cdot) = \frac{1}{2}(\cdot - \by
	)^{T}\bbq(\cdot - \by)$ and $g_{i} = \delta_{[-1 , 1]}$, $i = 1 , 2 , \ldots ,
	100$, problem \eqref{RandomPB} fits our model \eqref{mainproblem}. We generate random
	instances of problem \eqref{RandomPB} as follows.
	\begin{itemize}
		\item Set $d = 100$ and $n = 200$.
		\item Generate $\bbx \in \real^{n \times d}$ with standard normal entries.
		\item \ed{Set $\bbq = \frac{1}{n}\bbx^T \bbd^2 \bbx$ where $\bbd = \diag(1/n^2,1/(n-1)^2,...,1)$}
		\item Generate $\by$ with standard normal entries.
	\end{itemize}
	We compare the standard Conditional Gradient (CG) algorithm \cite{J2013}, the	Random Block
	Conditional Gradient (RBCG) algorithm \cite{LJJSP2013} \ed{(with block chosen uniformly at random)}, the Cyclic Block Conditional Gradient \ed{where the order of the blocks is fixed for all iterations (CBCG-C) and the Cyclic Block Conditional Gradient where the order of the blocks is chosen by a random permutation at each iteration (CBCG-P)}\delete{, which is the method of interest in this paper}. For each algorithm, we
	compare three different stepsize rules:
	\begin{itemize}
		\item {\bf Predefined:} the stepsize which is given in \eqref{pre-defined}.
			Note that for RBCG, there is a slight modification in the definition of the predefined stepsize. \ed{There is no notion of a ``cycle" in RBCG and the stepsize has the form $2n / (\tilde{k} + 2n)$ where $\tilde{k}$ is the number of blocks updated so far, see \cite{LJJSP2013}}.
		\item {\bf \ed{Adaptive with Backtracking}:} set $\bba = \bbi$ and compute the stepsize using the backtracking
			scheme.
		\item {\bf Exact line search:} the stepsize is chosen to minimize the objective as in
			\eqref{exact}. \ed{Note that as explained in Section \ref{SSec:Exact}, this corresponds to use the adaptive stepsize (\ref{adaptive}) with $\bba = \bbd \bbx / \sqrt{n}$ since we consider a quadratic objective function.}
	\end{itemize}
	For problems of the form \eqref{RandomPB}, the second strategy is questionable since a
	closed form expression is available for the exact line search. However, the comparison to this strategy illustrates
	differences in algorithmic behaviors related to the choice of $\bba$ in the model
	\eqref{mainproblem} \ed{as well as the performance of the backtracking scheme}.
	
	\begin{figure}[t]
		\centering
		\includegraphics[width=.9\textwidth]{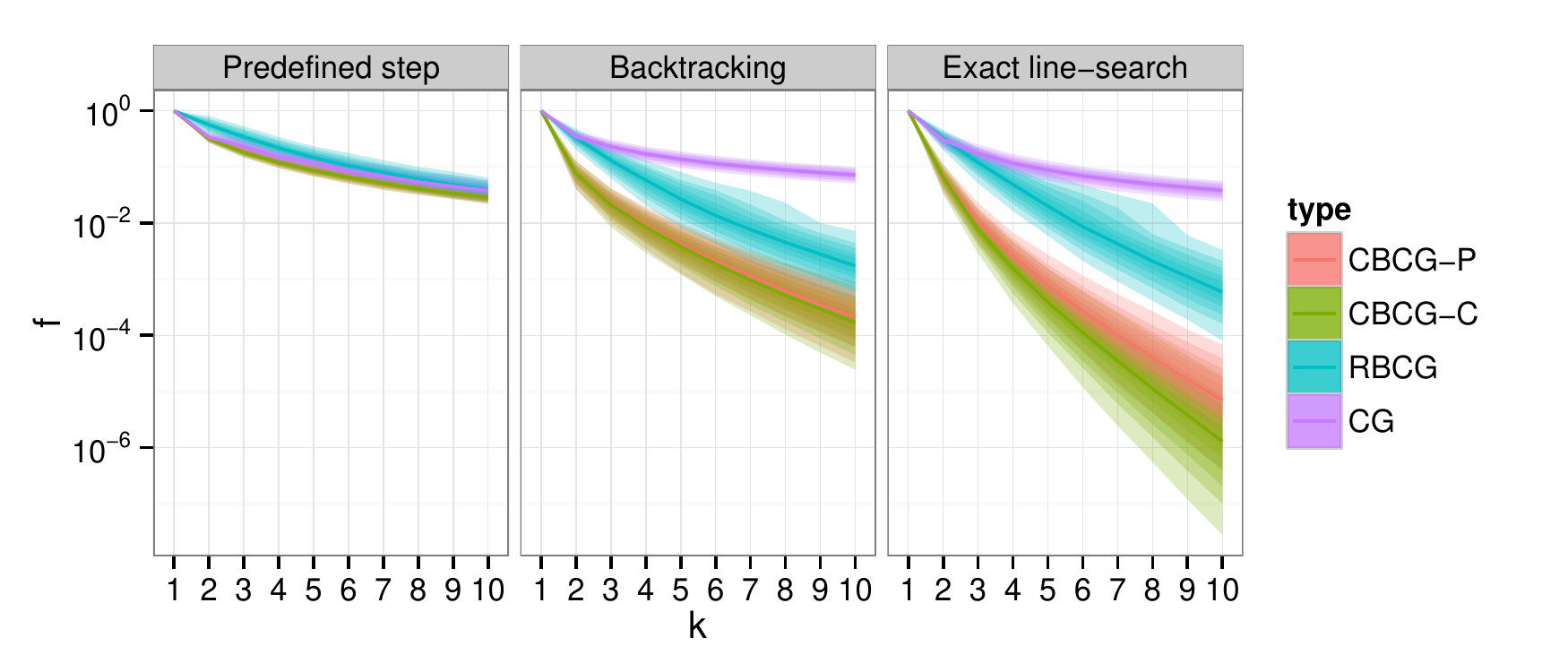}
		\caption{Comparison of Conditional Gradient (CG), its random block version (RBDG), \ed{its cyclic block version with fixed block order (CBCG-C) and its cyclic block version with random permutation block order (CBCG-P)}. We compare three different stepsize strategies based on
			$1000$ randomly generated instances of problem \eqref{RandomPB}. The central line
			is the median over the 1000 runs and the ribbons show 98\%, 90\%, 80\%, 60\% and
			40\% quantiles. For all methods, $k$ represents the number of effective passes through $d$ coordinates.}
		\label{fig:firstResults}
	\end{figure}

	\begin{remark} \label{rem:lsAdapt}
		For CG and RBCG, a sublinear convergence rate holds for the three stepsize strategies.
		For \ed{both CBCG-C and CBCG-P}, in the case of the predefined and of the adaptive stepsizes, convergence is
		ensured by Theorems \ref{T:RatePre} and \ref{T:RateBack}, respectively. Furthermore,
		since we consider quadratic objective functions, exact line search stepsize strategy
		can be seen as a particular case of our adaptive strategy and the convergence follows from Theorem
		\ref{T:RateAdap} (see also Section \ref{SSec:Exact}).
	\end{remark}
	We generate $1000$ random instances of problem  \eqref{RandomPB}, $f_w$ denoting the objective function of the $w$-th randomly generated problem, $w = 1, 2, \ldots, 1000$.  For each problem, we run the \ed{four} different algorithms with the three different stepsize rules (the initialization is at the origin). The results for the first $10$ iterations are presented in Figure \ref{fig:firstResults}. For each algorithm, increasing $k$ by $1$ means that $N = 100$ blocks have been queried, randomly for RBCG, sequentially for CBCG and all at once for CG. Since each objective function is generated randomly, it does not make sense to directly compare performance across different problems on the same graph. To overcome this, for each $w = 1,2 , \ldots, 1000$, we ``center'' and ``scale'' the function values. That is, for each objective $f_w$, $w = 1, 2, \ldots, 1000$, we estimate the optimal value $f_w^*$ of (\ref{RandomPB}) by running CBCG-P with exact line search for 200 more iterations. For such a number of iterations, we observed on preliminary experiments that the algorithm reached machine precision on random instances of problem (\ref{RandomPB}). The quantity plotted in Figure \ref{fig:firstResults} is given by the following affine transformation,
\begin{align*}
    \frac{f_w(\bx^k) - f_w^*}{f_w(\bo) - f_w^*},
\end{align*}
so that in Figure \ref{fig:firstResults}, the first value is always $1$ and the represented quantities are positive and asymptotically tend to $0$. The main comments regarding this experiment are the following:

	\begin{itemize}
		\item For each stepsize rule, CBCG-C has an advantage.
		\item \ed{There is not much difference between the two variants CBCG-C and CBCG-P in this experiment}.
		\item The predefined stepsize rule leads to much slower convergence.
		\item \ed{Adaptive with backtracking} and exact line search rules yield improved convergence speed for
			both CBCG and RBCG, which is not the case with the CG.
		\item Exact line search rule has a slight advantage over the backtracking rule.
	\end{itemize}
	The last point deserves further comments. Indeed, in model (\ref{mainproblem}), there are different possible choices of matrix $\bba$ and function $F$ that lead to equivalent problems. Despite being equivalent problems, different model choices may lead to variations in algorithmic performances when numerically solving a problem. This is what we observe here. Indeed, as pointed out in Section  \ref{SSec:Exact}, the exact line search strategy corresponds to choosing an $F$ that is perfectly conditioned. This means that variations of $F$ around its minimum are isotropic. On the other hand, choosing $\bba = \bbi$ leads to a choice of $F$ that is less well-behaved. These numerical experiments suggest that choosing $\bba$ that corresponds to a better conditioned $F$ leads to better numerical performances. \ed{Finally, the gap between the two strategies remains small, highlighting the efficiency of the backtracking scheme.}

\subsection{\ed{Structured} SVM} \label{SSec:Numerics-SVM}
	The main motivation for the introduction of block version of CG method with random update
	rule in \cite{LJJSP2013} is that it leads to a new efficient algorithm for training the
	\ed{structured} SVM \cite{TGK2004,TJHA2005}. We refer the reader to \cite{LJJSP2013} and the
	references therein for background on this problem and its relations to the conditional
	gradient method. In brief, the \ed{structured} SVM solves a multi-class classification task. It is
	dedicated to problems for which the output classes are embedded in a combinatorial
	structure such as trees, graphs or sequences. In this setting the number of classes can be
	enormous which results in optimization problems with an untractable number of linear
	constraints. For some of these problems efficient decoding algorithms can be used as
	oracles to compute sub-gradients of the \ed{structured} SVM problem. They can also be used as oracles to solve the linearized sub-problem required to run the conditional gradient and block conditional gradient methods on the dual of the \ed{structured} SVM. In this section, we briefly recall the mathematical formulation of the dual
	\ed{structured} SVM problem and provide a numerical comparison between random and cyclic update
	rules for block conditional gradient in this context. The purpose is not to be exhaustive here and we solely focus on the aspects of the problem related to optimization.
	Therefore, we compare the numerical performances of the two selection rules on this real world example based on an optimization criterion. In
	this section, $N$ denotes the number of training examples, $M$ denotes the number of output
	classes and $d$ is an integer such that $\real^{d}$ is a space of tractable size (such that
	elements of $\real^{d}$ can be stored in memory). The dual variable of the \ed{structured} SVM
	is a matrix $\balpha \in \real^{N \times M}$ with non-negative entries (this could actually
	be refined with example dependent output classes, but we stick to this notation as a first
	approximation). The dual problem of the \ed{structured} SVM can be written as follows
	\begin{equation} \label{StructuralSVM}
		\min_{\alpha \geq \bo} \left\{ \frac{\lambda}{2}\norm{\bba \balpha}^{2} - \mathrm{Tr}
		(\bb^{T} \balpha) : \, \balpha 1^{M} = 1^{N} \right\},
	\end{equation}
	where $\bba : \real^{N \times M} \rightarrow \real^{d}$ is a linear map, $\bb \in \real^{N
	\times M}$ is a matrix, $\mathrm{Tr}$ denotes the trace operator and $1^{s}$ denotes the
	vector in $\real^{s}$ which all entries are $1$. Problem \eqref{StructuralSVM} has an
	interesting product structure. Indeed, its feasible set can be viewed as a product of $N$
	simplices of dimension $M$. In the context of \ed{structured} SVM it is not necessary to store $
	\balpha$ explicitly, it is sufficient to store and update $\bba \balpha \in \real^{d}$ and $
	\mathrm{Tr}(\bb^{T}\balpha)$. With this information, we can use the specific decoding
	oracles to solve simplex constrained linear subproblems and run the conditional gradient
	algorithm. This constitutes the main advantage of the method here, it allows to explore a
	potentially very large space by using only implicit conditional gradient steps (recall that
	$M$ is the size of a set of combinatorial nature, see \cite{LJJSP2013} for a complete
	derivation and more details). We use the code provided by the authors of \cite{LJJSP2013}
	which gives the possibility to train the \ed{structured} SVM using RBCG, \ed{CBCG-C (fixed block order) or CBCG-P (random permutation block order)} on the Optical
	Character Recognition (OCR) originally proposed in \cite{TGK2004}. \delete{We consider random
	update rule and cyclic update rule with varying block order (or random permutation). } We
	only consider the exact line search strategy \eqref{exact}. Note that our convergence analysis
	can be applied in this setting (see also Section \ref{SSec:Exact}). \ed{It can be checked (see \cite{LJJSP2013} for details about the structured SVM model) that the constant $C_2$ given by Theorem \ref{T:RateAdap} remains bounded away from zero as $N$ grows, in the setting of model \eqref{StructuralSVM}. Therefore, the rate given in \eqref{T:RateAdap:1} suffers from a multiplicative dependence in $N$. From a theoretical point of view, in the context of large training sets (large $N$), this problematic dependence leaves room for improved convergence analysis.} Numerical results in
	terms of the optimality measure $S$ for various values of $\lambda$ are given in Figure
	\ref{fig:structSVM}. \delete{The global behavior is similar to what we observed in the synthetic
	examples of Section \ref{SSec:Numerics-Synthetic}. In particular, the cyclic update rule
has a slight advantage. } \ed{The global behavior differs from what we have observed in the synthetic examples of Section \ref{SSec:Numerics-Synthetic}. The random block selection rule (RBCG) has an advantage over the purely cyclic block selection rule (CBCG-C). The random permutation approach (CBCG-P) gives the best performances here. Since our analysis is the same for both CBCG-C and CBCG-P, it does not allow to explain this empirical difference.}
	\begin{figure}[t]
		\centering
		\includegraphics[width=.9\textwidth]{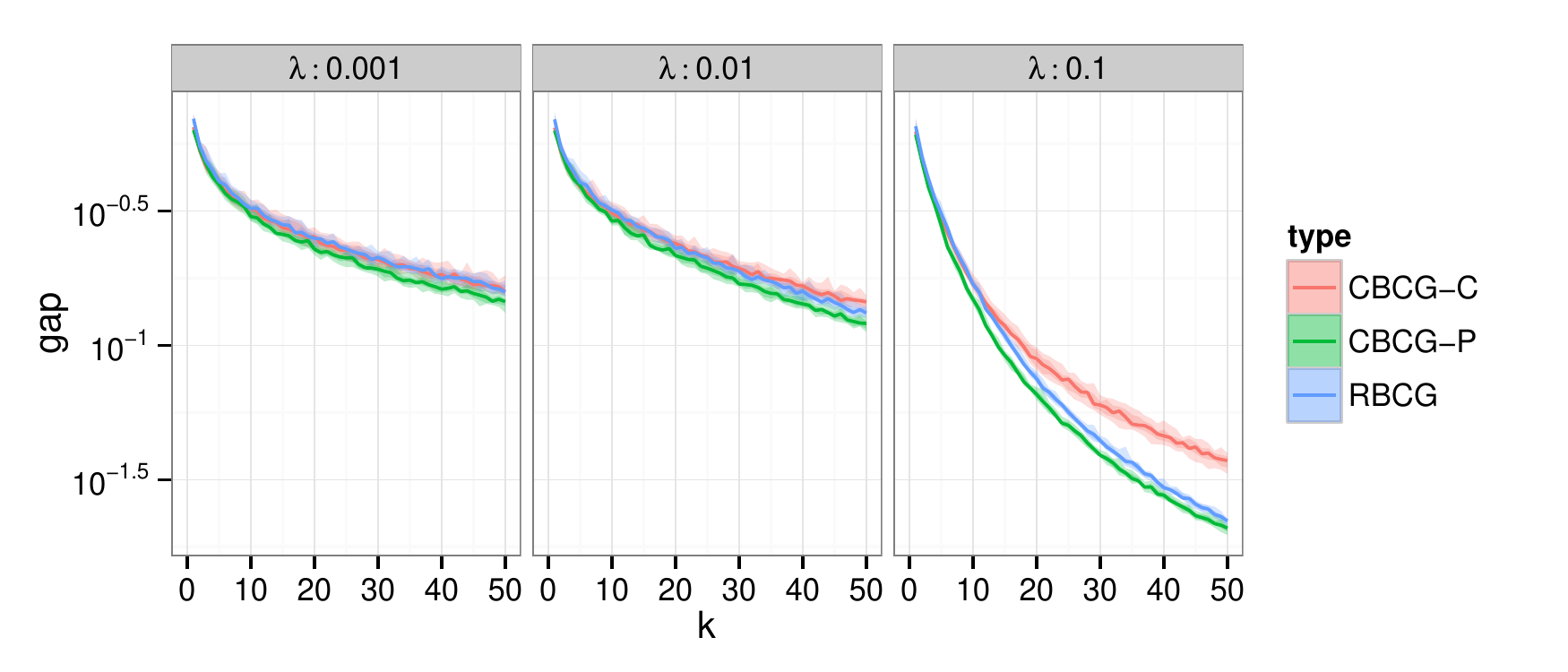}
		\caption{Comparison of RBCG and \ed{CBCG-C (purely cyclic) and CBCG-P (random permutation)} for the \ed{structured} SVM training on the OCR dataset
			of \cite{TGK2004}. The quantity represented is the optimality measure defined by
			\eqref{D:Px} and the heading represents the value of $\lambda$ (from (\ref{StructuralSVM})). The central line is
			the median over $20$ runs and the ribbons show 80\% and 50\% quantiles. For all methods, $k$ represents the number of effective passes
through all the blocks (which correspond to datapoints here).}
		\label{fig:structSVM}
	\end{figure}
	%
	
\section{Discussion and Future Work}
\subsection{\ed{Comparison with existing results}}
We have described the CBCG algorithm and provided an explicit sublinear convergence rate estimate of the form $O(1/k)$, which is canonical for linear oracle based algorithms (up to multiplicative factor) \cite{P87}. This asymptotic rate cannot be improved in general \cite{canon1968tight}. \ed{A few words on the constants are in order. For the traditional CG method, the multiplicative constant is proportional to $L_FD^2$ \cite{J2013}\footnote{The constant presented in \cite{J2013} relates to an affine invariant notion of curvature, but it can be upper bounded by the term we consider.}. In order to compare these rates with those obtained for block decomposition methods, we consider in this discussion that one iteration of a block decomposition method consists in updating $N$ arbitrarily chosen blocks (which is consistent with the notation in this paper). The multiplicative constant proposed in \cite{LJJSP2013} for the average case complexity of RBCG is proportional to $(L_FD^2 + H(\bx^0) - H^*)$ independently of the number of blocks\footnote{As in the case of CG, the result is not presented with this exact constant, but it is an upper bound that we use for discussion purposes.}. The independence of the constant with respect to the number of blocks is an important property that we do not recover for the cyclic block updating rule. This behavior is consistent with previous observations comparing random and cyclic update rules \cite{BTet2013} where multiplicative dependance in the number of  blocks also appeared in the convergence rate analysis of different types of methods.

For the sake of clarity, we will discuss the results of section \ref{Sec:Analysis} when $\bba= \bbi$, all $D_i$ are equal and $\beta_i = L_F$, $i=1,2,\ldots,N$. For the predefined stepsize, the quantity $C_1/L_FD^2$ ($C_1$ given by (\ref{L:PreMain:1})) grows like $\Omega(\sqrt{N})$. Furthermore, for the adaptive stepsize, the quantity $NC_2/L_FD^2$ ($C_2$ given by (\ref{L:MainAda:1})) grows like $\Omega(N^2)$. This is much worse than the $\sqrt{N}$ dependance of the predefined strategy although practical simulations tend to show that the adaptive rule is much faster. Therefore, there is a room for improvement in the analysis and this raises the natural question of the possibility to get multiplicative constants that do not depend on the number of blocks (which is in general the case for random block selection rules). As we have seen in Sections \ref{Sec:ML} and \ref{Sec:Numerics}, for SDCA and structured SVM applications, the effect is to have a multiplicative dependance of the convergence rate in the size of the training set. Therefore, although theoretically interesting, the rates we derive are limited to explain performances on large training sets (large number of blocks in the dual) which is the motivation for using block decomposition methods in machine learning applications.
}
	
\delete{However, contrary to the average case complexity estimate of RBCG \cite{LJJSP2013}, the exact expression of the rate is not a direct generalization of that of CG algorithm. Recall that for traditional CG method, the multiplicative constant is proportional to $L_FD^2$ \cite{J2013}. The constants given in Theorems \ref{T:RatePre} and \ref{T:RateAdap} feature additional multiplicative and additive terms. In particular the quantity $C_1/L_FD^2$, for $C_1$ given by (\ref{L:PreMain:1}) or the quantity $NC_2/L_FD^2$ for $C_2$ given by (\ref{L:MainAda:1}), show multiplicative dependence in the number of blocks $N$. This behavior is consistent with previous observations comparing random and cyclic update rules \cite{BTet2013}. It is expected that the analysis of cyclic rules lead to worse constants since they represent worst case complexity analysis. An important theoretical question is whether this can be leveraged or not for cyclic rules. In other words, is it possible to prove an explicit convergence rate of the form $M/k$ for the cyclic rule such that $M/(L_FD^2)$ does not depend on the number of blocks $N$?}

For practical applications however, this remark should be mitigated since we are comparing upper bounds. These bounds \ed{may only reflect limitations of the analysis, not of the methods,} and their comparison may not shed much light on the comparative behavior of different rules on practical problems. Indeed, our numerical experiments on synthetic and real-world examples reproducibly suggest an advantage of \ed{cyclic or random permutation variants} over fully random update rule. This is again something that has already been observed in other contexts \cite{BTet2013,shalev2013stochastic}. A question related to the discussion of the previous paragraph is to give a theoretical justification to this observation or eventually provide different iteration dependent block update rules that comply with it (see for example \cite{nutini2015coordinate} for an illustration in the context of gradient descent).

\subsection{\ed{Future Directions}}
Finally, a natural question is that of the extension of specificities of the conditional gradient method in our block decomposition setting. Potential directions include the following:
\begin{itemize}
	\item {\bf Linear convergence.} CG is known to converge linearly when the optimum is in the relative interior of the feasible set \cite{guelat1986some,lacoste2013affine} or when the feasible set is strongly convex and the gradient of the smooth part of the objective is non-zero on the feasible set \cite{LP1966}.
	\item {\bf Dual interpretation of the block decomposition.} Generalized CG is known to implicitly generate subgradient sequences \cite{B2012} related to the mirror descent algorithm \cite{BT03} applied to a dual problem. Similarly, a dual interpretation of RBCG is in terms of stochastic subgradient \cite{LJJSP2013}.
	\item {\bf Generalization of the results to exact line search stepsize strategies.} The analysis of such stepsize strategies is not problematic for CG or RBCG \cite{J2013,LJJSP2013}. However, we could not generalize it for CBCG, except in the quadratic case (see Section \ref{SSec:Exact}), and thus developing an analysis for the exact line search strategy is in our opinion an important task. \ed{Most of our analysis collapses without further assumptions, because it is no longer possible to get sufficient decrease conditions of the type of Lemma \ref{L:Suff}. We therefore expect that a different path needs to be considered.}
	\item \ed{{\bf Generalization to inexact oracles.} In many practical applications, the search direction given by the oracle is computed by an algorithm. It is therefore relevant to consider multiplicative errors in order to use well defined stopping criteria for the inner algorithm. This was already considered in the random block variant \cite{LJJSP2013}.}
\end{itemize}
\appendix
\section{Proof of Lemma \ref{L:BlockDescent}} \label{sec:app}
We adapt the standard proof of the descent Lemma, see e.g., \cite[Proposition A.24]{B99}.  Under the assumptions of the lemma, using the fundamental theorem of calculus for line integrals on the segment $[\bx, \bx + \bbu_i \bh]$, we have
	\begin{align}
		\label{eq:CBDL1}
		f(\bx + \bbu_i \bh) &= f(\bx) + \int_0^1 \left\langle\nabla f(\bx + t \bbu_i \bh), \bbu_i \bh \right\rangle\,dt\nonumber\\
		&= f(\bx) + \left\langle\nabla f(\bx), \bbu_i \bh \right\rangle + \int_0^1\left\langle\nabla f(\bx + t \bbu_i \bh) - \nabla f(\bx), \bbu_i \bh \right\rangle\,dt.
	\end{align}
	We can bound the integrand term for any $t \in [0,1]$ as follows
	\begin{align}
		\label{eq:CBDL2}
		\left\langle\nabla f(\bx + t \bbu_i \bh) - \nabla f(\bx), \bbu_i \bh \right\rangle\nonumber&=\left\langle \bba^T\nabla F(\bba(\bx + t \bbu_i \bh)) - \bba^T\nabla F( \bba \bx), \bbu_i \bh \right\rangle\nonumber\\
		&= \left\langle\nabla F(\bba(\bx + t \bbu_i \bh)) - \nabla F(\bba \bx), \bba \bbu_i \bh \right\rangle\nonumber\\
		&= \left\langle\nabla F(\bba \bx + t \bba_i \bh) - \nabla F(\bba \bx), \bba_i \bh \right\rangle\nonumber\\
		&\leq \left\|\nabla F(\bba \bx + t \bba_i \bh)) - \nabla F(\bba \bx)\right\| \cdot \left\|\bba_i \bh \right\|\nonumber\\
		&\leq t\beta_i ||\bba_i \bh||^2,
	\end{align}
	where we have used Cauchy-Schwartz inequality and Assumption  \ref{AssumptionB} to obtain the last two inequalities. Combining (\ref{eq:CBDL1}) and (\ref{eq:CBDL2}), we have
	\begin{align*}
		f(\bx + \bbu_i \bh) &\leq  f(\bx) + \left\langle\nabla f(\bx), \bbu_i \bh \right\rangle + \beta_i ||\bba_i \bh||^2 \int_0^1 t\;dt \\
		&= f(\bx) + \left\langle\nabla_i f(\bx), \bh \right\rangle + \frac{\beta_i ||\bba_i \bh||^2}{2},
	\end{align*}
	which proves the desired result.\hfill$\Box$
\bibliographystyle{abbrv}
\bibliography{notes-2}
\end{document}